\documentclass[a4paper]{amsart}

\usepackage{amsmath,amssymb}
\usepackage{amsthm}
\pdfoutput=1

\usepackage{graphicx}
\usepackage{bm}
\usepackage{tikz-cd}
\usepackage{mathtools}
\usepackage{comment}
\usepackage{bigdelim, multirow}
\usepackage{latexsym}
\usepackage{stmaryrd}
\usepackage{braket}
\usepackage{color}
\usepackage{setspace}
\usepackage{comment}
\usepackage{url}
\newcommand{\Z}{\mathbb{Z}}

\newcommand{\R}{\mathbb{R}}
\newcommand{\C}{\mathbb{C}}
\newcommand{\Q}{\mathbb{Q}}
\newcommand{\T}{\mathbb{T}}

\newcommand{\gerh}{\mathfrak{h}}
\newcommand{\gerk}{\mathfrak{k}}
\newcommand{\gersl}{\mathfrak{sl}}
\newcommand{\gert}{\mathfrak{t}}

\newcommand{\calO}{\mathcal{O}}

\newcommand{\calH}{\mathcal{H}}

\newcommand{\om}{\omega}
\newcommand{\Om}{\Omega}

\newcommand\vep{\varepsilon}
\newcommand\vphi{\varphi}
\newcommand{\Cstar}{\C^*}

\newcommand\Hom{\mathop{\mathrm{Hom}}\nolimits}

\newcommand\an{\mathop{\mathrm{an}}\nolimits}

\newcommand\Aut{\mathop{\mathrm{Aut}}\nolimits}

\newcommand\codim{\mathop{\mathrm{codim}}\nolimits}

\newcommand\Image{\mathop{\mathrm{Im}}\nolimits}
\newcommand\id{\mathop{\mathrm{id}}\nolimits}

\newcommand\Ker{\mathop{\mathrm{Ker}}\nolimits}
\newcommand\Lie{\mathop{\mathrm{Lie}}\nolimits}

\newcommand\Mat{\mathop{\mathrm{Mat}}\nolimits}
\newcommand\Proj{\mathop{\mathrm{Proj}}\nolimits}
\newcommand\rank{\mathop{\mathrm{rank}}\nolimits}
\newcommand\reg{\mathop{\mathrm{reg}}\nolimits}

\newcommand\Sing{\mathop{\mathrm{Sing}}\nolimits}

\newcommand\Span{\mathop{\mathrm{Span}}\nolimits}
\newcommand\Spec{\mathop{\mathrm{Spec}}\nolimits}

\def\fr<#1/#2>{\frac{#1} {#2}}

\newcommand{\LargeO}{\text{{\huge{0}}}}

\newcommand{\vcbig}[1]{\multicolumn{2}{c}{$\mbox{\smash{\Huge $#1$}}$}}
\newcommand{\xrightarrowdbl}[2][]{%
  \xrightarrow[#1]{#2}\mathrel{\mkern-14mu}\rightarrow
}

\newtheorem{thm}{Dont use this}[section]
\newtheorem{theorem}[thm]{Theorem}
\newtheorem{proposition}[thm]{Proposition}
\newtheorem{corollary}[thm]{Corollary}
\newtheorem{lemma}[thm]{Lemma}
\newtheorem{fact}[thm]{Fact}

\newtheorem{problem}[thm]{Problem}
\newtheorem{question}[thm]{Question}
\newtheorem{assumption}[thm]{Assumption}

\newtheorem*{definition*}{Definition}
\newtheorem*{theorem*}{Theorem}
\newtheorem*{proposition*}{Proposition}
\newtheorem*{corollary*}{Corollary}
\newtheorem*{lemma*}{Lemma}
\newtheorem*{problem*}{Problem}
\newtheorem*{question*}{Question}
\newtheorem*{conjecture*}{Conjecture}

\newtheorem{definition and lemma}[thm]{Definition $\&$ Lemma}
\newtheorem{definition and proposition}[thm]{Definition $\&$ Proposition}
\newtheorem*{remark*}{Remark}
\theoremstyle{definition}
\newtheorem{definition}[thm]{Definition}
\newtheorem{example}[thm]{Example}

\newtheorem{remark}[thm]{Remark}

\theoremstyle{definition}

\theoremstyle{plain}
\title{%
The universal covers of hypertoric varieties and Bogomolov's decomposition
}
\author[T. Nagaoka]{Takahiro Nagaoka}
\address[Takahiro Nagaoka]{Department of Mathematics,
	Graduate School of Science,
	Kyoto University,
	Kyoto, 606-8522, Japan}
\email{tnagaoka@math.kyoto-u.ac.jp}

\date{}

\begin{document}
\pagestyle{plain}
\begin{abstract}
In this paper, we study the (singular) universal cover of an affine hypertoric variety. We show that it is given by another affine hypertoric variety, and taking the universal cover corresponds to taking the simplification of the associated hyperplane arrangement. Also, we describe the fundamental group of the regular locus of an affine hypertoric variety in general. 
In the latter part, we show that the hamiltonian torus action is block indecomposable if and only if $\Cstar$-equivariant symplectic structures on the associated hypertoric variety are unique up to scalar. 
In particular, we establish the analogue of Bogomolov's decomposition for hypertoric varieties, which is proposed by Namikawa for general conical symplectic varieties. As a byproduct, we show that if two affine (or smooth) hypertoric varieties are $\Cstar$-equivariant isomorphic as varieties, then they are also the hamiltonian torus action equivariant isomorphic as symplectic varieties. This implies that the combinatorial classification actually gives the classification of these varieties up to $\Cstar$-equivariant isomorphisms.   
\end{abstract} 
\maketitle
\setcounter{tocdepth}{1}
\tableofcontents

\section{Introduction}\label{sec:Intro}


Hypertoric varieties were introduced as a hyperk\"{a}hler analogue of toric varieties by Bielawski and Dancer (\cite{BD}), and they have been extensively studied by many authors (\cite{HS}, \cite{Kocohomology}, \cite{Kovariation}, \cite{PW}, etc). An (affine) hypertoric variety has a symplectic form $\om$ on its regular locus and admits a good $\Cstar$-action as $\om$ is {\it homogeneous} with some weight $\ell$, i.e., $s^*\om=s^\ell\om$. In general, such affine varieties $(Y, \om)$ are called {\it conical symplectic varieties}, and its symplectic resolution $\pi : (\widetilde{Y}, \widetilde{\om}) \to (Y, \om)$ (i.e., $\pi^*\om$ extends to a symplectic form $\widetilde{\om}$ on $\widetilde{Y}$) has been extensively studied by many authors from the view point of not only algebraic geometry but also geometric representation theory (cf.\ \cite{BPW1}, \cite{BLPW2}). 


In algebro-geomteric context, for conical symplectic varieties, their fundamental geometric properties, deformation theory, birational geometry are extensively studied. On the other hand, their finite coverings, in particular, the ``universal coverings'' which are also conical symplectic varieties defined below have not been well-studied (recently, finite covers of nilpotent orbit closures are studied in \cite{Nam:new}). In \cite{Nam:equi}, Namikawa proposed the following problem. 
\newtheorem*{problem:fundamental}{Problem \ref{problem:fundamental}}
\begin{problem:fundamental}{\rm (cf.\ \cite[Problem 7.2]{Nam:equi})}\\
For a conical symplectic variety $(Y, \om)$, is the fundamental group $\pi_1(Y_{\reg})$ of the regular part of $Y$ finite ?
\end{problem:fundamental}
\noindent In \cite{Nam:fund}, as a partial general answer, he proved the finiteness of the algebraic fundamental group $\hat{\pi}_1(Y_{\reg})$. If $\pi_1(Y_{\reg})$ is finite, one can consider the (singular) universal covering of $Y$ as follows (for the detail, see Definition \& Proposition \ref{def and prop:universal cover}).  
\newtheorem*{def and prop:universal cover}{Definition \& Proposition \ref{def and prop:universal cover}}
\begin{def and prop:universal cover}{\rm (The universal covers of conical symplectic varieties)}\\
Let $(Y, \om)$ be a conical symplectic variety with $m:=|\pi_1(Y_{\reg})|<\infty$. Then, there exists a unique conical symplectic variety $(\overline{Y}, \overline{\om})$ and a finite $\Cstar$-equivariant morphism $\vphi : (\overline{Y}, \overline{\om}) \to (Y, \om)$ 
such that its restriction $\vphi| : \vphi^{-1}(Y_{\reg}) \to Y_{\reg}$ gives the universal cover of $Y_{\reg}$. 
In this paper, we call $\overline{Y}$ the {\rm universal covering of $Y$}. 
\end{def and prop:universal cover}
\noindent By taking the universal cover, we can often obtain a new example of conical symplectic varieties (cf.\ \cite{Nam:new}). On the other hand, in general, for any conical symplectic variety $(Y, \om)$ and a finite subgroup $G$ of $\Aut^{\Cstar}(Y)$ preserving $\om$, $Y/G$ is also a conical symplectic variety (cf.\ \cite[Proposition 2.4]{Bea}). This implies that the universal cover is important in classification since any conical symplectic varieties is obtained by a finite quotient of the universal one. Anyway, it is natural to ask the following question: 
\newtheorem*{problem:universalcover}{Problem \ref{problem:universalcover}}
\begin{problem:universalcover}
For a given conical symplectic variety $(Y, \om)$ with $|\pi_1(Y_{\reg})|<\infty$, describe the universal cover $\vphi : (\overline{Y}, \overline{\om}) \to (Y, \om)$ and $\pi_1(Y_{\reg})$. 
\end{problem:universalcover}

\noindent Another motivation to study the universal cover is coming from the analogue of Bogomolov's decomposition as the following: 
\newtheorem*{problem:bogomolov}{Problem \ref{problem:bogomolov}}
\begin{problem:bogomolov}{\rm (\cite[Problem 7.3]{Nam:equi}, Bogomolov's decomposition)}\\
For any conical symplectic variety $(Y, \om)$ with $|\pi_1(Y_{\reg})|<\infty$, is its universal cover $(\overline{Y}, \overline{\om})$ decomposed into the product 
\[(\overline{Y}, \overline{\om})\cong\prod_{i}{(Y_i, \om_i)}\]
of {\it irreducible} conical symplectic varieties $(Y_i, \om_i)$, that is $\om_i$ is the unique conical symplectic structure on $Y_i$ up to scalar ? 
\end{problem:bogomolov}

In this paper, we study the universal cover of affine hypertoric varieties. Our main theorems give complete answers for the above problems for affine hypertoric varieties and give an interpretation of taking the universal cover in terms of combinatorics of the associated hyperplane arrangement (or matroid). Also, we will give an application to the classification of hypertoric varieties as byproducts.   
\vspace{5pt}

Now we set up notations to state our results precisely. 
Let $A\hspace{-2pt}=\hspace{-2pt}[\bm{a_1}, \ldots, \bm{a_n}]$ be a unimodular matrix of rank $d$ and take $B^T\hspace{-2pt}=\hspace{-2pt}[\bm{b_1}, \ldots ,\bm{b_n}]$ so that  the following is exact: 
\begin{equation*}\setlength{\abovedisplayskip}{0pt}
\begin{tikzcd}
0\ar[r]&\Z^{n-d}\ar[r, "{B}"] &\Z^{n}\ar[r, "{A}"]&\Z^d\ar[r]&0         
\end{tikzcd}.\setlength{\belowdisplayskip}{0pt}
\end{equation*}
Then by using the natural Hamiltonian $\T_\C^d$-action on $\C^{2n}$ induced from  $A^T: \T_\C^d\hookrightarrow \T_\C^n$, we define a hypertoric variety $Y_A(\alpha):=\C^{2n}/\hspace{-3pt}/\hspace{-3pt}/_\alpha \T_\C^d:=\mu^{-1}(0)/\hspace{-3pt}/_\alpha\T_\C^d$ as the symplectic reduction, where $\alpha\in\Z^d$ is a GIT parameter, and $\mu=\mu_A: \C^{2n}\to\C^d$ is the $\T_\C^d$-invariant moment map. By definition, we have a natural projective morphism $\pi_\alpha : Y_A(\alpha)\to Y_A(0)$. Then, $Y_A(0)$ is a conical symplectic variety, and for generic $\alpha$, $\pi_\alpha$ gives a projective symplectic resolution. 
For any projective toric variety, we can consider the associated polytope and read off many geometric properties from this polytope. Similarly, for any hypertoric variety $Y_A(\alpha)$, one can consider the associated hyperplane arrangement $\calH_B^\alpha:=\Set{H_i : \langle \bm{b_i}, -\rangle=-\widetilde{\alpha}_i}$ where $\widetilde{\alpha}$ satisfies $\alpha=A\widetilde{\alpha}$ ($\calH_B^{\alpha}$ is well-defined up to translations). We consider $\calH_B^{\alpha}$ as a multiset of hyperplanes.  

In section 5, we consider the universal cover of an affine hypertoric variety $Y_A(0)$. 
\vspace{1pt}

\renewcommand{\arraystretch}{1}\arraycolsep=2pt
\noindent First, we can assume {\setlength\arraycolsep{1pt}$B^T\hspace{-5pt}=$\raise1ex\hbox{$\begin{array}{rcccccccl}
&\multicolumn{3}{c}{\smash{\overbrace{\hspace{6ex}}^{\ell_1}}}&&\multicolumn{3}{c}{\smash{\overbrace{\hspace{6ex}}^{\ell_s}}}&\\
\ldelim[{1}{3pt}[]&\bm{b^{(1)}},&\cdots,&\bm{b^{(1)}},&\cdots,&\bm{b^{(s)}},&\cdots,&\bm{b^{(s)}}&\rdelim]{1}{0.1pt}[]\\
\end{array}$} \ , where if $k_1 \neq k_2$, then $\bm{b^{(k_1)}} \hspace{-2pt} \neq \hspace{-2pt} \pm\bm{b^{(k_2)}}$. Then, we consider the {\it simplification} $\overline{B}^T \hspace{-4pt} :=[\bm{b^{(1)}}, \ldots,  \bm{b^{(s)}}]$} of $B^T$, and take $\underline{A}=[\bm{\underline{a}_1}, \ldots, \bm{\underline{a}_s}]$ satisfying the following exact sequence (cf.\ section \ref{sec:universal}):  
\begin{equation*}
\begin{tikzcd}
0\ar[r]&\Z^{n-d}\ar[r, "\overline{B}"]\ar[d, equal]&\Z^{s} \ar[r, "\underline{A}" ]\ar[d, hook, "B_0"]&\Z^{d-(n-s)}\ar[r]\ar[d, hook, "i"]&0\\
0\ar[r]&\Z^{n-d}\ar[r, "B"]&\Z^n\ar[r, "A"]&\Z^d\ar[r]&0
\end{tikzcd}, \ \ \ 
\begin{tikzcd}
\T_\C^s&\T_\C^{d-(n-s)}\ar[l, hook', "\underline{A}^T"']\\
\T_\C^n\ar[u, two heads, "B_0^T"']&\T_\C^d\ar[l, hook', "{A}^T"']\ar[u, two heads, "i^*"']
\end{tikzcd}.
\end{equation*}
Then, we call $Y_{\underline{A}}(0)$ the {\it simplification}  of $Y_A(0)$. Actually, from the combinatorial point of view, taking simplification corresponds to replacing all multiplicated hyperplanes of $\calH_B^0$ by a single one and obtaining $\calH_{\overline{B}}^0$ as the following.  
\begin{figure}[h]
\begin{center}
\includegraphics[width=200pt, height=40pt]{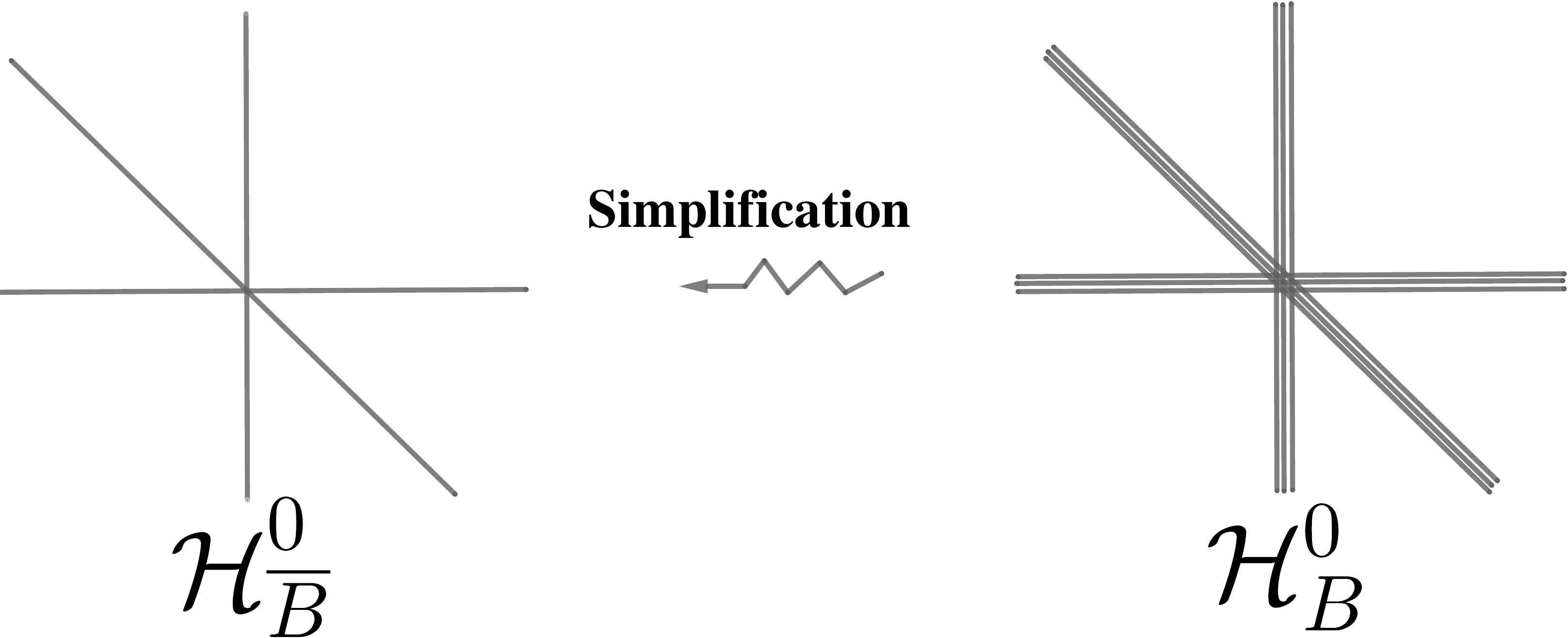}
\end{center}
\end{figure}

\noindent Now, consider the following embeddings: 
\[\begin{tikzcd}\Gamma:=\prod_{k=1}^s{\Z/\ell_k\Z}\ar[r, hook]& \T_\C^s& \T_\C^{d-(n-s)} \ar[l, hook', "\underline{A}^T"'] \end{tikzcd},\]
where we consider $\Gamma$ as a multiplicative subgroup of $\T_\C^s$ in the usual way. Then we have the answers for Problem \ref{problem:fundamental} and Problem \ref{problem:universalcover} as the following. 

\newtheorem*{thm:fundamentalgrp}{Theorem \ref{thm:fundamentalgrp} \& Proposition \ref{prop:explicit}}
\begin{thm:fundamentalgrp}
In the above setting, for $Y_A(0)$ and its simplification $Y_{\underline{A}}(0)$, there exists the following {$\Cstar$-equivariant} commutative diagram:  
\[\begin{tikzcd}
Y_{\underline{A}}(0)\ar[r, "\vphi"]\ar[d]&Y_A(0)\\
Y_{\underline{A}}(0)/(\Gamma/\Gamma\cap\T_\C^{d-(n-s)})\ar[ur, "\cong"']
\end{tikzcd}.\]
Moreover, {$\vphi$ preserves symplectic structures,} 
and $\Gamma/\Gamma\cap\T_\C^{d-(n-s)}$ acts on $\vphi^{-1}(Y_A(0)_{\reg})$ freely. In particular, we have 
\[\pi_1(Y_A(0)_{\reg})\cong\Gamma/\Gamma\cap\T_\C^{d-(n-s)}=\Gamma/q(\Ker \widetilde{B}^T),\]
where $\widetilde{B}^T:=[m_1\bm{b^{(1)}}, \cdots, m_s\bm{b^{(s)}}]$, $q : \Z^s \to \Gamma$, and $m_k:=\prod_{i\neq k}{\ell_i}$. 
\end{thm:fundamentalgrp}
\noindent This theorem says that the universal cover of an affine hypertoric variety $Y_A(0)$ is obtained by the simplification $Y_{\underline{A}}(0)$. As a corollary, one can characterize when $Y_A(0)_{\reg}$ is simply-connected as the following. 
\newtheorem*{cor:converse}{Corollary \ref{cor:converse}}
\begin{cor:converse}
The following are equivaelnt: 
\begin{itemize}
\item[(i)] $\codim \Sing(Y_A(0))\geq4$. 
\item[(ii)] $\pi_1(Y_A(0)_{\reg})=0$.
\item[(iii)] $Y_A(0) (\text{or} \ \calH_B^0)$ is simple, i.e., $ \ell_k=1$ for any $k$. 
\end{itemize}
\end{cor:converse}
\noindent As we will note in Remark \ref{rem:nontrivialgeneral}, the equivalence ``(i) $\Leftrightarrow$ (ii)'' is very special, and it doesn't hold for general conical symplectic varieties. 
\vspace{2pt}

In the latter part, we consider Bogomolov's decomposition for (not necessarily simple) affine hypertoric varieties. In general, we can assume that $A$ is of the form $A=[O_p | A_1\oplus\cdots\oplus A_r]$, where each $A_m$ is indecomposable and $O_p$ is the $d\times p$-zero matrix. We say $A$ is {\it block indecomposable} if $p=0$ and $r=1$. Then, we have the following: 

\newtheorem*{thm:main final}{Theorem \ref{thm:main final}}
\begin{thm:main final}{\rm (Bogomolov's decomposition for affine hypertoric varieties)}\\
If $A=[O_p | A_1\oplus\cdots\oplus A_r]$, then the following gives a decomposition into irreducible conical symplectic varieties: 
\[(Y_A(0), \om) \cong \prod_{i=1}^p{(\C^2, dz\wedge dw)}\times\prod_{m=1}^r{(Y_{A_m}(0), \om_m)}.\]
In particular, $A$ is block indecomposable if and only if $Y_A(0)$ is an irreducible conical symplectic variety. In this case, any (not necessarily symplectic) homogeneous 2-form on $Y_A(0)_{\reg}$ with weight 2 is unique up to scalar. 
\end{thm:main final}

\noindent As we will note in Corollary \ref{cor:generalhypertoric}, the same statement holds for any hypertoric varieties $Y_A(\alpha)$. The key steps of the proof of the uniqueness of conical symplectic structures are twofolds. First, through the universal covering morphism $\vphi : Y_{\underline{A}}(0) \to Y_A(0)$, we can reduce to the case of $Y_{\underline{A}}(0)$. Next, we will show that $\mu_{\underline{A}}^{-1}(0)$ is a normal complete intersection and smooth in codimension 3 (cf.\ Proposition \ref{prop:keyreflexive}). Then, by Vetter's criterion on the reflexiveness of $\Om^p_{Z}$ for any normal complete intersection $Z$, we can conclude that any homogeneous 2-form on $Y_{\underline{A}}(0)_{\reg}$ comes from a homogeneous 2-form on $\mu_{\underline{A}}^{-1}(0)$.

In the final section, we determine the space of homogeneous 2-forms on decomposable hypertoric varieties (cf.\ Proposition \ref{prop:generalform}). As an application, we show the following refined version of the classification result in \cite[{Theorem 4.2}]{Nag}. 

\newtheorem*{cor:classification}{Corollary \ref{cor:classification}}
\begin{cor:classification}(cf.\ Remark \ref{rem:combinatorial})\\
For any two affine hypertoric varieties $Y_A(0)$ and $Y_{A'}(0)$, 
the following are equivalent: 
\begin{itemize}
\item[(i)] $Y_A(0)$ and $Y_{A'}(0)$ are $\Cstar$-equivariant isomorphic as algebraic varieties. 
\item[(ii)] $(Y_A(0), \om)$ and $(Y_{A'}(0), \om')$ are $\Cstar\times\T_\C^{n-d}$-equivariant isomorphic as symplectic varieties, where $\T_\C^{n-d}$-action is the remaining Hamiltonian torus action. 
\item[(iii)] $A\sim A'$ (cf.\ Definition \ref{def:operation}). In other words, $M(A)\cong M(A')$ as matroids, where $M(A)$ is the associated vector matroid. 
\end{itemize}
The same statement holds for smooth hypertoric varieties, where we replace (iii) by  \\
(iii') The zonotope tilings obtained from the hyperplane arrangements $\calH_B^{\alpha}$ and $\calH_{B'}^{\alpha'}$ are same in the sense of \cite{AP}.   
\end{cor:classification}
This corollary can be seen as the (weaker) hypertoric analogue of Berchtold's theorem \cite{Ber} on toric varieties, which says that for any two toric varieties, they are isomorphic as abstract varieties if and only if they are isomorphic as toric varieties. In particular, this means that the combinatorial classification (of the associated {\it fans}) actually gives the classification of toric varieties as abstract varieties. 


This paper is organized as follows. In section 2, we review the definition of conical symplectic varieties $(Y, \om)$, and we define the universal cover of them. We also give a sufficient condition that $Y_{\reg}$ is simply-connected (cf.\ Proposition \ref{prop:codim4fundamental}). In section 3, we review hypertoric varieties and give some examples. In section 4, for the later discussion, we recall a result on the stratification of affine hypertoric varieties $Y_A(0)$. In section 5, we prove that the universal cover of $Y_A(0)$ is given by simplification $Y_{\underline{A}}(0)$. Moreover, we describe $\pi_1(Y_A(0)_{\reg})$. In section 6, we give a concrete computation of $\pi_1(Y_A(0)_{\reg})$ and a necessary and sufficient condition that $\pi_1(Y_A(0)_{\reg})$ is trivial. In Remark \ref{rem:alternative}, we also note an alternative possible  way to the computation. In section 7, we establish the Bogomolov's decomposition for affine hypertoric varieties $Y_A(0)$. 
In section 8, we determine the space of homogeneous 2-forms on a general $Y_A(0)$. Then, we give an application to the classification of hypertoric varieties.  

\medskip

\noindent\textbf{Acknowledgements.}
The author wishes to express his gratitude to his supervisor Yoshinori Namikawa for stimulating discussions. He is also greatly indebted to Hiraku Nakajima for giving some comments on his presentation, which leads to start this project. He is grateful to Ryo Yamagishi for spending much time to discuss with him. The author is also grateful to Nicholas Proudfoot for sharing his nice idea on a computation of the fundamental group (in Remark \ref{rem:alternative}). He also wishes to express his thanks to Masahiko Yoshinaga for letting him know the reference \cite{Bj}. He is also grateful to Makoto Enokizono for pointing out an error in the previous version. The author is partially supported by Grant-in-Aid for JSPS Fellows 19J11207.



\section{Conical symplectic varieties and the universal coverings}
In this section, we introduce conical symplectic varieties, and we note that for any conical symplectic variety $(Y, \om)$ with $|\pi_1(Y_{\reg})|<\infty$, its (possibly singular) universal cover is also a conical symplectic variety. Second, we prove that if $Y$ is smooth in codimension 3, i.e., $\codim \Sing(Y)\geq4$, and $Y$ admits a  symplectic resolution, then $Y_{\reg}$ is simply-connected (cf. Corollary \ref{cor:hypertoric trivial}). 
\vspace{2pt}

A pair $(Y, {\om})$ of a normal algebraic variety $Y$ and {a} 2-form ${\om}$ on the regular locus $Y_{\reg}$ is called a {\it symplectic variety} if ${\om}$ is symplectic, and there exists (or equivalently, for any) a resolution $\pi: \widetilde{Y}\to Y$ such that the pull-back $\pi^*\om$ of ${\om}$ extends to {an} algebraic 2-form $\widetilde{\om}$ on $\widetilde{Y}$, i.e., $Y$ has only canonical singularities. Moreover, a resolution $\pi: (\widetilde{Y}, \widetilde{\om}) \to (Y, \om)$ is called a {\it symplectic resolution} if $\widetilde{\om}$ is also symplectic. These definitions are due to Beauville \cite{Bea}. 

Now, we define conical symplectic varieties as the following: 
\begin{definition}{\rm (Conical symplectic variety)}\label{def:conical}\\
An affine symplectic variety $(Y\hspace{-2pt}=\hspace{-2pt}\Spec R, \ {\om})$ with $\Cstar$-action (called {\it conical $\Cstar$-action}) is called a {\it conical symplectic variety} if it satisifies the following: 
\begin{itemize}
\item[(i)]The grading induced from the $\Cstar$-action to the coordinate ring $R$ is positive, i.e., $R=\bigoplus_{i\geq0}{R_i}$ and $R_0=\C$. 
\item[(ii)]${\om}$ is {\it homogeneous} with respect to the $\Cstar$-action, i.e., there exists $\ell\in\Z$ (the {\it weight} of ${\om}$) such that $s^*{\om}=s^{\ell}{\om} \ (s\in\Cstar)$. 
\end{itemize}
For a given two conical symplectic varieties, we say that they are isomorphic as conical symplectic varieties if there exists an $\Cstar$-equivariant isomorphism between them preserving symplectic structures.   
\end{definition}  
\begin{remark}
As noted in \cite[Lemma 2.2]{Nam:equi}, the weight $\ell$ is always positive. 
\end{remark}

There are many examples of conical symplectic varieties, for example, the (normalization of) nilpotent orbit closures in semisimple Lie algebras, the Slodowy slices, quiver varieties, hypertoric varieties, symplectic quotient singularities, and so on. 
Next, we define Poisson varieties which include symplectic varieties as special cases. 
\begin{definition}A pair $(Y, \{-, -\})$ of a variety and a skew-symmetric bilinear {morphism} $\{-, -\}: \calO_{Y}\times\calO_{Y}\to\calO_{Y}$ is a {\it Poisson variety} if it satisfies the following: 
\begin{itemize}
\item[(1)]$\{f, gh\}=\{f, g\}h+\{f, h\}g$,
\item[(2)]$\{f, \{g, h\}\}+\{g, \{h, f\}\}+\{h, \{f, g\}\}=0$.
\end{itemize}
\end{definition}
For a symplectic variety $(Y, \om)$, one can consider a natural Poisson structure on $Y_{\reg}$ as $\{f, g\}:=\om(H_f, H_g)$, where $H_f$ is the Hamiltonian vector ($f\in\calO_{Y_{\reg}}$), and by the normality of $Y$, it uniquely extends to a Poisson structure on $Y$. 




Below, for a conical symplectic variety $(Y, \om)$, we will consider the universal cover of $Y$. To do this, it is natural to ask the following question. 
\begin{problem}{\rm (\cite[Problem 7.2]{Nam:equi})}\label{problem:fundamental}\\
For a conical symplectic variety $(Y, \om)$, is the fundamental group $\pi_1(Y_{\reg})$ of the regular part of $Y$ finite ?
\end{problem}
Namikawa gave a general partial answer for this problem as the following. 
\begin{theorem}{\rm (\cite{Nam:fund})}\\
In the above setting, the algebraic fundamental group $\hat{\pi}_1(Y_{\reg})$ is finite. 
\end{theorem}

If we know the finiteness of the fundamental group of the regular locus of a given conical symplectic variety $Y$, we can consider the (possibly singular) {\it universal covering of $Y$} as follows. 
Although this construction has already appeared in \cite[p. 512]{Nam:equi} and \cite[p.4]{Nam:new} briefly, we will give a proof in detail. 

\begin{definition and proposition}{\rm (The universal cover of conical symplectic varieties)}\label{def and prop:universal cover}\\
Let $(Y, \om)$ be a conical symplectic variety with $m:=|\pi_1(Y_{\reg})|<\infty$. Then, there exists a unique conical symplectic variety $(\overline{Y}, \overline{\om})$ and a finite $\Cstar$-equivariant morphism $\vphi : (\overline{Y}, \overline{\om}) \to (Y, \om)$ (with respect to $\sigma_{\overline{Y}}$ and $\sigma_{Y}^m$, where the conical $\Cstar$-action is denoted by $\sigma_{\overline{Y}}$ (resp. $\sigma_Y$)) such that its restriction $\vphi| : \vphi^{-1}(Y_{\reg}) \to Y_{\reg}$ gives the universal cover of $Y_{\reg}$: 
 \[\begin{tikzcd}[contains/.style = {draw=none,"\in" description,sloped}, icontains/.style = {draw=none,"\ni" description,sloped}, inclusion/.style = {draw=none,"\subset" description,sloped}, niclusion/.style = {draw=none,"\supset" description,sloped}]
(\overline{Y}, \overline{\om})\ar[r, niclusion]\ar[d, "\vphi"]&\vphi^{-1}(Y_{\reg})\ar[d, "\vphi|"]\\
(Y, \om)\ar[r, niclusion]&Y_{\reg}
\end{tikzcd}.\]
In this paper, we call $\overline{Y}$ the {\rm universal covering of $Y$}. 
\end{definition and proposition}
\begin{proof}
First, we can take the universal cover $Z_0 \to Y_{\reg}$, then $Z_0$ is a complex manifold. By the finiteness of the fundamental group $\pi_1(Y_{\reg})$ and Riemann's existence theorem \cite[Th\'{e}or\`{e}me 5.1]{SGA1}, 
there exists an algebraic variety $\overline{Y}_0$ and a finite \'{e}tale morphism $\vphi_0 : \overline{Y}_0 \to Y_{\reg}$ such that $\overline{Y}_0^{\an}=Z_0$ and $\vphi_0$ is the original universal cover. Then, by considering the spectrum $\overline{Y}:=\Spec \overline{R}$ of the integral closure of the coordinate ring $R:=\C[Y]$ in the function field $\C(\overline{Y}_0)$, $\vphi_0$ uniquely extends to a finite surjective morphism $\vphi : \overline{Y} \to Y$, in particular we have the following diagram: 
\[\begin{tikzcd}[contains/.style = {draw=none,"\in" description,sloped}, icontains/.style = {draw=none,"\ni" description,sloped}, inclusion/.style = {draw=none,"\subset" description,sloped}, niclusion/.style = {draw=none,"\supset" description,sloped}]
\overline{Y}\ar[r, niclusion]\ar[d, "\vphi"]&\overline{Y}_0=\vphi^{-1}(Y_{\reg})\ar[d, "\vphi|=\vphi_0"]\\
Y\ar[r, niclusion]&Y_{\reg}
\end{tikzcd}.\]

Next, we define a conical $\Cstar$-action $\sigma_{\overline{Y}} : \Cstar\times \overline{Y} \to \overline{Y}$. Consider the following diagram:    
\[\begin{tikzcd}
\Cstar\times \overline{Y}_0\ar[d, "\id\times\vphi_0"]&\overline{Y}_0\ar[d, "\vphi_0"]\\
\Cstar\times Y_{\reg}\ar[r, "\sigma_Y^m"]&Y_{\reg}
\end{tikzcd}.\]
Since the induced homomorphism $(\sigma_Y^m\circ (\id \times \vphi_0))_* : \Z\cong\pi_1(\Cstar\times \overline{Y}_0) \to \pi_1(Y_{\reg})$ is a zero homomorphism, we can uniquely lift the morphism $\sigma_Y^m\circ (\id \times \vphi_0)$ to an analytic morphism $\sigma_{\overline{Y}_0} : \Cstar\times \overline{Y}_0 \to \overline{Y}_0$ satisfying $\sigma_{\overline{Y}_0}(1, y)=y$ for any $y\in\overline{Y}_0$. Then, one can check that $\sigma_{\overline{Y}_0}$ gives a $\Cstar$-action on $\overline{Y}_0$. 
Since for a $\C$-scheme $W$, the category of \'{e}tale finite coverings over $W$ is equivalent to the category of complex analytic finite coverings over $W$ (cf. \cite[Th\'{e}or\`{e}me 5.1]{SGA1}), 
we can easily show that $\sigma_{\overline{Y}_0}$ is an algebraic morphism. Then, we can uniquely extend this $\Cstar$-action $\sigma_{\overline{Y_0}}$ to a $\Cstar$-action $\sigma_{\overline{Y}}$ on $\overline{Y}$ since we have $\codim_{\overline{Y}}{(\overline{Y}-\overline{Y}_0)}\geq 2$ by the construction and the normality of $\overline{Y}$. Moreover, one can easily check that this action $\sigma_{\overline{Y}}$ is also positive (cf.\ Definition \ref{def:conical}). In fact, note that $\vphi : \overline{Y} \to Y$ is $\Cstar$-equivariant with respect to $\sigma_{\overline{Y}}$ and $\sigma_Y^m$. Then, through the injective homomorphism $\vphi^* : R=\bigoplus_{i\in\Z_{\geq0}}{R'_i}\hookrightarrow \overline{R}=\bigoplus_{i\in\Z}{\overline{R}_i}$ as graded algebras, $\overline{R}$ is integral over $R$, where the gradings correspond to the action $\sigma_Y^m$ and $\sigma_{\overline{Y}}$ respectively. In this situation, if we have an element $f\in\overline{R}_d$ of a negative degree $d<0$, then there exists a relation as $f^N+r_1f^{N-1}+\cdots+r_N=0$, where $r_i\in R'_{di}$. Since we have $R'_{di}=0$ and $\overline{R}$ is domain, we obtain $f=0$. In the similar way, one can easily show that $\overline{R}_0=\C$ since $\C$ is algebraically closed.      

Finally, we will prove that $\overline{Y}$ is a conical symplectic variety. Since $\vphi_0$ is an \'{e}tale $\Cstar$-equivariant morphism, $\vphi_0^*\om$ is a conical symplectic form on $\overline{Y}_0$. Moreover, from $\codim_{\overline{Y}}{(\overline{Y}-\overline{Y}_0)}\geq2$, $\vphi_0^*\om$ will uniquely extend to a conical symplectic form $\overline{\om}$ on $(\overline{Y})_{\reg}$ (in detail, for example see the proof of \cite[{Proposition 2.15}]{Nag}). Now, note that $Y$ has symplectic singularities, so $Y$ has only canonical singularities. Then, by $\codim_{\overline{Y}}{(\overline{Y}-\overline{Y}_0)}\geq2$, $\vphi$ is \'{e}tale in codimension 1, so $\overline{Y}$ has also only canonical singularities (cf.\ \cite[Proposition 5.20]{KM}). Thus, by definition of canonical singularities, we can conclude that $(\overline{Y}, \overline{\om})$ is a symplectic variety. It is clear that $\overline{\om}$ is homogeneous of weight $m\ell$ with respect to the $\Cstar$-action $\sigma_{\overline{Y}}$. Thus, $(\overline{Y}, \overline{\om})$ is a conical symplectic variety and satisfies the desired commutative diagram. 

\end{proof}

Since the universal cover often gives a new example of conical symplectic varieties, it is natural to consider the following problem. 
\begin{problem}\label{problem:universalcover}
For a given conical symplectic variety $(Y, \om)$ with $|\pi_1(Y_{\reg})|<\infty$, describe the universal cover $\vphi : (\overline{Y}, \overline{\om}) \to (Y, \om)$ and $\pi_1(Y_{\reg})$. 
\end{problem}

Below, we will show that if a conical symplectic variety $Y$ with $\codim \Sing(Y)\geq4$ admits a symplectic resolution, then its regular locus is simply-connected, i.e., it is already the universal covering. First, we note the following general result.  
\begin{theorem}{\rm (\cite[Theorem 1.1]{Tak})}\label{thm:takayama}\\
Let $Y$ be a normal variety and assume that $(X, \Delta)$ is a klt pair for some divisor $\Delta$. Then, for any resolution $f : \widetilde{Y}\to Y$, the induced homomorphism $f_{*} : \pi_1(\widetilde{Y}) \to \pi_1(Y)$ is an isomorphism.  
\end{theorem}

Using this, we can show the following. 
\begin{proposition}\label{prop:codim4fundamental}
Assume that a conical symplactic variety $Y$ satisfies $\codim_Y{\Sing (Y)}\geq4$ and admits a symplectic resolution $f: \widetilde{Y} \to Y$．Then, $\pi_1(Y_{\reg})\cong \pi_1(\widetilde{Y})\cong\pi_1(Y)=0$, in particular, $Y_{\reg}$ is simply-connected.  
\end{proposition}
\begin{proof}
{Since $f| : f^{-1}(Y_{\reg}) \to Y_{\reg}$ is a crepant birational morphism, it is an isomorphism, in particular, } $\pi_1(Y_{\reg})=\pi_1(f^{-1}(Y_{\reg}))$. 
On the other hand, since any symplectic resolution is semi-small (cf.\ \cite[Lemma 2.7]{Kalstratify}, \cite[Theorem 3.2]{Fu:survay}), $2\codim_{\widetilde{Y}}f^{-1}(\Sing(Y))\geq\codim_Y{\Sing(Y)}\geq4$. Hence, $\codim_{\widetilde{Y}}f^{-1}(\Sing(Y))\geq2$. In particular, $\pi_1(\widetilde{Y})\cong\pi_1(Y_{\reg})$. Since a conical symplectic variety has canonical singularities, $\pi_1(Y_{\reg})\cong \pi_1(\widetilde{Y})\cong\pi_1(Y)$ by the above general theorem. Also, a conical symplectic variety is contractible by the conical action, so we complete the proof. 
\end{proof}
\begin{remark}
By \cite{Nam:terminal}, for any symplectic variety $(Y, \om)$, $Y$ is terminal if and only if $\codim_Y{\Sing(Y)}\geq4$. 
\end{remark}

\section{Hypertoric varieties}

In this section, we review the definition of hypertoric varieties and give some examples. We basically use the same notation as \cite{HS} (except the notation of hypertoric varieties). Let $A: \Z^n\xrightarrowdbl{}N\cong\Z^d$ be a surjective linear map to a free abelian group $N$ (in this paper, we almost always fix a basis of $N$, identify $N$ with $\Z^d$ and consider $A=[\bm{a_1}, \ldots, \bm{a_n}]$ as a $d\times n$-matrix). Then, take $B^T=[\bm{b_1}, \ldots ,\bm{b_n}]\in \Mat_{(n-d)\times n}(\Z)$ as the following is exact: 
\[\begin{tikzcd}
0\arrow{r}&\Z^{n-d}\arrow{r}{B} &\Z^{n}\arrow{r}{A}&N\cong\Z^d\arrow{r}&0         
\end{tikzcd}.\]
The configuration $\{\bm{b_1}, \ldots ,\bm{b_n}\}$ in $\Z^{n-d}$ is called a {\it Gale dual} of $\{\bm{a_1}, \ldots ,\bm{a_n}\}$. 

\begin{assumption}\label{ass:1}
Throughout this paper, we assume that $A$ is unimodular, i.e., any $d\times d$-minor is $\pm1$. Moreover, to consider only essential cases, we assume that for any $j$, {$\bm{b_j}\neq\bm{0}$}. As we note in \cite[Corollary 2.30]{Nag}, we can always replace general $B$ by such ones without changing the associated hypertoric variety. 
\end{assumption}

By applying $\Hom(-, \Cstar)$ to the above exact sequence, we obtain the following exact sequence of algebraic tori. 
\[\begin{tikzcd}
1\ar[r]&\T_\C^d\ar[r, "{A^T}"]&\T_\C^n\ar[r, "{B^T}"]&\T_\C^{n-d}\ar[r]&1
\end{tikzcd},\] 
where $\T_\C^k:=(\Cstar)^k$. Then, through the embedding $\T_\C^d\overset{A^T}{\hookrightarrow}\T_\C^n$ and the natural action $\T_\C^n\curvearrowright(T^*\C^n, \omega_\C)=(\C^{2n}, \sum_{j=1}^n{dz_j\wedge dw_j})$, we obtain a hamiltonian $\T_\C^d$-action on $(\C^{2n}, \omega_\C)$ (for the definition of the hamiltonian action, we refer \cite[Definition 9.43]{Kir}). More explicitly, this action is described as the following: 
\[\bm{t}\cdot(z_1, \ldots ,z_n, w_1, \ldots ,w_n):=(\bm{t}^{\bm{a_1}}z_1, \ldots ,\bm{t}^{\bm{a_n}}z_n, \bm{t}^{\bm{-a_1}}w_1, \ldots ,\bm{t}^{\bm{-a_n}}w_n),\]
where $\bm{t}^{\bm{a_j}}:=t_1^{a_{1j}}\cdots t_d^{a_{dj}}$. 
Since this action is hamiltonian, we have a $\T_\C^d$-invariant moment map $\mu=\mu_A: \C^{2n}\to(\gert_\C^d)^*=\C^d$ (more strongly, $\mu$ is $\T_\C^n$-invariant) as follows: 
\[\mu(\bm{z}, \bm{w}):=\mu_A(\bm{z}, \bm{w}):=\sum_{j=1}^nz_jw_j\bm{a_j}.\]
Then, the hypertoric variety is defined as the GIT quotient of $\mu^{-1}(0)$ by $\T_\C^d$, and actually we can describe this explicitly as the following (cf. \cite{HS}). 
For $\alpha\in\Z^d=\Hom(\T_\C^d, \Cstar)$ and the coordinate ring $\C[\mu^{-1}(0)]$ of $\mu^{-1}(0)$, we set 
\[\C[\mu^{-1}(0)]^{\T_\C^d, \alpha}:=\{f\in\C[\mu^{-1}(0)] \ | \ f(\bm{t}\cdot m)=\alpha(\bm{t})f(m) \ \forall \bm{t}\in\T_\C^d\}.\]
An element $p\in \mu^{-1}(0)$ is called {\it $\alpha$-semistable} if there exists some $k>0$ and $f\in\C[\mu^{-1}(0)]^{\T_\C^d, k\alpha}$ such that $f(p)\neq0$, and we denote the set of $\alpha$-semistable elements by $\mu^{-1}(0)^{\alpha-ss}$. Also, an element $p\in \mu^{-1}(0)^{\alpha-ss}$ is called {\it $\alpha$-stable} if the stabilizer group of $\T_\C^d$ at $p$ is finite and $\T_\C^d\cdot p\subseteq \mu^{-1}(0)^{\alpha-ss}$ is closed. We denote the set of $\alpha$-stable elements by $\mu^{-1}(0)^{\alpha-st}$.

\begin{definition}(Hypertoric varieties, cf.\ \cite{HS})\\ 
For $\alpha\in\Z^d=\Hom(\T_\C^d, \Cstar)$, we consider a graded algebra $\bigoplus_{k\in\Z_{\geq0}}{\C[\mu^{-1}(0)]^{\T_\C^d, k\alpha}}$. Then, we define $Y_A(\alpha)$ as the following: 
\[Y_A(\alpha):=\Proj\left(\bigoplus_{k\in\Z_{\geq0}}{\C[\mu^{-1}(0)]^{\T_\C^d, k\alpha}}\right)=\mu^{-1}(0)^{\alpha-ss}/\hspace{-3pt}/\T_\C^d, \]
where $/\hspace{-3pt}/$ denotes the categorical quotient.  
We also often write $Y_A(\alpha)=\C^{2n}/\hspace{-3pt}/\hspace{-3pt}/_{\alpha}\T_\C^d$. We call $Y_A(\alpha)$ a {\it hypertoric variety}. In particular, if $\alpha=0$, we simply denote $Y_A(0)=\C^{2n}/\hspace{-3pt}/\hspace{-3pt}/\T_\C^d$ and call $Y_A(0)$ an {\it affine hypertoric variety}. 
\end{definition}
\begin{remark}
By definition, $Y_A(\alpha)$ admits the remaining hamiltonian torus action $\T_\C^{n-d}=\T_\C^n/\T_\C^d\curvearrowright Y_A(\alpha)$. Also, we consider a natural $\Cstar$-action on $\C^{2n}$ such that $s\cdot(z_1, \ldots ,z_n, w_1, \ldots ,w_n)=(s^{-1}z_1, \ldots ,s^{-1}z_n, s^{-1}w_1, \ldots ,s^{-1}w_n)$. Since this action commutes with $\T_\C^d$-action and $\mu$ is $\Cstar$-equivariant, a $\Cstar$-action on $Y_A(\alpha)$ is induced. In particular, for any $\alpha$, the natural projective morphism $\pi_\alpha : Y_A(\alpha) \to Y_A(0)$ is $\Cstar\times \T_\C^{n-d}$-equivariant. 
\end{remark}
\begin{remark}\label{rem:semistable affine}
By definition, for any $\alpha\in\Z^d$, $\mu^{-1}(0)^{\alpha-ss}$ is covered by $\T_\C^d$-invariant affine open subsets. Moreover, one can explicitly describe $\mu^{-1}(0)^{\alpha-ss}$ as the following (cf.\ \cite[Lemma 3.4]{Kosurvey}):  
\begin{equation*}\tag{*}
\mu^{-1}(0)^{\alpha-ss}=\Set{(\bm{z}, \bm{w}) \in \mu^{-1}(0) \ | \ \alpha \in \sum_{i: z_i \neq0}{\Q_{\geq0}\bm{a_i}} + \sum_{i:w_i\neq0}{\Q_{\geq0}(-\bm{a_i})}}.
\end{equation*}
\end{remark}

We recall some well-known basic properties of hypertoric varieties. Below, $\alpha\in\Z^d$ is called {\it generic} if $\alpha$  is not contained in any proper subspace generated by some $\bm{a_j}$'s. 
\begin{proposition}{\rm (for example, see \cite[{Theorem 2.16}]{Nag})}\\
Assume that $\alpha\in\Z^d$ is generic. Then, $\mu^{-1}(0)^{\alpha-ss}=\mu^{-1}(0)^{\alpha-st}$ and the torus action $\T_\C^d\curvearrowright \mu^{-1}(0)^{\alpha-st}$ is free. 
Also, the hypertoric variety $(Y_A(\alpha), \om)$ is a smooth symplectic variety, where $\om$ is the standard symplectic structure which is induced from the standard symplectic structure $(\C^{2n}, \om_\C)$.  
Moreover, the natural projective morphism $\pi_\alpha : Y_A(\alpha) \to Y_A(0)$ gives a $\Cstar\times\T_\C^{n-d}$-equivariant symplectic resolution, and $Y_A(0)$ is a conical symplectic variety with weight 2.   
\end{proposition}
\begin{remark}\label{rem:hypertoric is symplectic}
For any (not necessarily generic) $\alpha\in\Z^d$, $Y_A(\alpha)$ is also a symplectic variety, and the natural projective morphism $\pi_\alpha : Y_A(\alpha) \to Y_A(0)$ is an isomorphism on $\pi_\alpha^{-1}(Y_A(0)_{\reg})$. In fact, first, $\pi_\alpha$ is birational by the similar argument of \cite[{Proposition 2.13}]{Nag}. Next, we can take a generic $\alpha_{\text{gen}}\in\Z^d$ such that $\mu^{-1}(0)^{\alpha_{\text{gen}}-ss}\subset \mu^{-1}(0)^{\alpha-ss}$. Then, this inclusion gives a natural birational morphism $\pi_{\alpha_{\text{gen}}, \alpha} : Y_A(\alpha_{\text{gen}}) \to Y_A(\alpha)$ such that $\pi_{\alpha_{\text{gen}}}=\pi_{\alpha_{\text{gen}}, \alpha}\circ \pi_\alpha$. Moreover, by the argument in the proof of \cite[{Proposition 2.15}]{Nag}, $\pi_{\alpha_{\text{gen}}}$ (resp. $\pi_{\alpha_{\text{gen}}, \alpha}$) is an isomorphism on $\pi_{\alpha_{\text{gen}}}^{-1}(Y_A(0))$ (resp. $\pi_{\alpha_{\text{gen}}, \alpha}^{-1}(Y_A(\alpha)_{\reg})$), in particular, $Y_A(\alpha)$ is a symplectic variety.  
\end{remark}

\begin{definition}{\rm (An equivalence relation between $A$'s (resp. $B$'s))}\label{def:operation}\\
For rank $d$ matrices $A$, $A'\in\Mat_{d\times n}(\Z)$, we say that $A'$ is {\it equivalent} to $A$ if there exists a $P\in GL_d(\Z)$ and an $n\times n$ signed permutation matrix $D$ such that $A'=PAD$, i.e., $A'$ is obtained from $A$ by a sequence of some elementary row operations over $\Z$, interchanging some column vectors $\bm{a_i}$'s, and multiplying some $\bm{a_i}$'s by $-1$. In this case, we denote $A\sim A'$. Similarly, we denote $B\sim B'$ if there exists $Q\in GL_{n-d}(\Z)$ and a $n\times n$ signed permutation matrix $D$ such that $B'=DBQ$.  
\end{definition}
Through this paper, we freely use these transformations of $A$ and $B$ since we have the following (this is essentially considered in \cite[Proposition 3.2]{AP}). 
\begin{lemma}{\rm (\cite[{Lemma 2.25}]{Nag})}\label{lem:fund}\\
In the setting above, if $A\sim A'$ (more explicitly $A'=PAD$), then there exists a natural following $\Cstar\times\T_\C^{n-d}$-equivariant isomorphism as symplectic varieties.  
\[ (Y_A(\alpha), \om) \xrightarrow{\sim} (Y_{A'}(\alpha'), \om'),\]
where $\alpha':=P\alpha$. 
\end{lemma}
Now, we give typical examples of hypertoric varieties. 

\begin{example}($A_{\ell-1}$-type surface singularities $S_{A_{\ell-1}}$)\label{ex:Atypesurface}\\
Consider the following exact sequence. 
\[\begin{tikzcd}[row sep=large, column sep=20ex, ampersand replacement=\&]
0\ar[r]\&\Z \ar[r, "{B_{\ell-1}=\begin{pmatrix}1\\1\\\vdots\\1\end{pmatrix}}"] \&\Z^\ell \ar[r, "{A_{\ell-1}=\begin{pmatrix} 1& & 0&-1\\ &\ddots&&\vdots\\ 0&&1&-1\\\end{pmatrix}}"] \&\Z^{\ell-1}\ar[r]\&0
\end{tikzcd}\]
For a generic $\alpha$, the corresponding symplectic resolution $\pi_\alpha : Y_A(\alpha) \to Y_A(0)$ is given by the following.  
\[
\begin{tikzcd}
Y_A(\alpha)\ar[r, "\sim"]\ar[d, "\pi_\alpha"]&\widetilde{S}_{A_{\ell-1}}\ar[d]&&:\text{the minimal resolution}\\
Y_A(0)\ar[r, "\sim"]&S_{A_{\ell-1}}:\ar[r, equal]&\left\{\det\begin{pmatrix}u&x\\y&u^{\ell-1}\end{pmatrix}=0\right\}&:\text{the $A_{\ell-1}$ type surface singularity}
\end{tikzcd},
\]
where the isomorphism $Y_A(0) \cong S_{A_{\ell-1}}$ is given by the following: 
\[\begin{tikzcd}
\C[\C^2/\Z_\ell]=\C[u, x, y]/\langle u^\ell-xy\rangle\ar[d, "\wr"]\\
\C[Y_{A_{\ell-1}}(0)]=\C[\mu^{-1}_{A_{\ell-1}}(0)]^{\T_\C^{\ell-1}}=\C[z_1w_1, \ldots, z_\ell w_\ell, \prod_{i=1}^\ell{z_i}, \prod_{i=1}^\ell{w_i}]/\langle z_1w_1=\cdots=z_\ell w_\ell\rangle
\end{tikzcd},\]
where $u\mapsto \ell z_iw_i$, $x\mapsto \prod_{i=1}^{\ell}{\sqrt{\ell} z_i}$, $y\mapsto \prod_{i=1}^{\ell}{\sqrt{\ell} w_i}$. One can easily check that this isomorphism $(\C^2/\Z_\ell, \om_{\text{st}})\cong(Y_{A_{\ell-1}}(0), \om)$ is an isomorphism as conical symplectic varieties, where $\Cstar$-action on $\C^2/\Z_\ell$ is induced from the natural scalar multiplication on $\C^2$, and $\om_{\text{st}}$ is induced from the standard symplectic form on $\C^2$. 
\end{example}

\begin{example}(The minimal nilpotent orbit closure $\overline{\calO}^{\min}_{A_{s-1}}$ of type $A_{s-1}$)\label{ex:minnilpotent}\\
Consider the following exact sequence. 
\[
\begin{tikzcd}[row sep=large, column sep=1ex, ampersand replacement=\&]
0\ar[r]\&[5ex]\Z^{s-1} \ar[r, "{B=\begin{pmatrix} 1 & &&\\ &\ddots&\vcbig{0}\\  &  &\ddots&\\ \vcbig{0}   &   &1\\ -1&\cdots&\cdots&-1\end{pmatrix}}"] \&[25ex] \Z^{s} \ar[r, "{A=\begin{pmatrix}1&1&\cdots&1\end{pmatrix}}"] \&[25ex]\Z\ar[r]\&[5ex] 0
\end{tikzcd}
\]
For a generic $\alpha>0$, one can prove directly that the corresponding symplectic resolution $\pi_\alpha : Y_A(\alpha) \to Y_A(0)$ is given by the following.
\[
\begin{tikzcd}[column sep=tiny, ampersand replacement=\&]
Y_A(\alpha)\ar[r, equal]\ar[d, "\pi_\alpha"]\&\left\{(\bm{z}, \bm{w}) \in \C^{2s} \ \left|\right. \ \bm{z}\neq0, \ \sum_{j=1}^{s}{z_jw_j}=0\right\}/\T_\C^1\ar[r, "\sim"]\&T^*\mathbb{P}^{s-1}\ar[d, "\pi'"]\\
Y_A(0)\ar[r, "\sim"]\&\Set{C\in \gersl_{s}(\C) \ | \ \text{All $2\times2$-minors of $C=0$}}\ar[r, equal]\&\overline{\calO}_{A_{s-1}}^{\min}
\end{tikzcd},
\]
where $\pi'$ is the Springer resolution. Concretely, $\pi_\alpha$ is given by the following. 
\[
\pi_\alpha(\bm{z}, \bm{w}):=\begin{pmatrix}z_1w_1&z_1w_2&\cdots&z_1w_{s}\\ w_1z_2&z_2w_2&\ddots&\vdots\\ \vdots&\ddots&\ddots&z_{s-1}w_{s}\\ w_1z_{s}&\cdots&w_{s-1}z_{s}&z_{s}w_{s} \end{pmatrix}.
\]

\end{example}

\begin{example}(A variant $\overline{\calO}^{\min}(\ell_1, \ell_2, \ell_3)$ of the minimal nilpotent closure $\overline{\calO}^{\min}_{A_{2}}$)\label{ex:Ominlll}\\
Consider the following $B$:
\[B=\begin{array}{rccll}
	\ldelim({11}{4pt}[] &1&0&\rdelim){11}{4pt}[]& \rdelim\}{3}{10pt}[$\ell_1$]\\
	&\vdots&\vdots& & \\
	&1&0& & \\
	&0&1& & \rdelim\}{3}{10pt}[$\ell_2$] \\
	&\vdots&\vdots& & \\
	&0&1& & \\
	&-1 &-1 & & \rdelim\}{3}{10pt}[$\ell_3$]\\
	&\vdots&\vdots& & \\
	&-1&-1&& \\
	\end{array} \ \ \ \ .\]
We denote by $\overline{\calO}^{\min}(\ell_1, \ell_2, \ell_3)$ the corresponding 4-dimensional affine hypertoric variety. When $\ell_1=\ell_2=\ell_3=1$, this is the same as $\overline{\calO}^{\min}_{A_2}$. 
The coordinate ring $\C[\overline{\calO}^{\min}(\ell_1, \ell_2, \ell_3)]$ can be computed as the following: 

\[\C[\overline{\calO}^{\min}(\ell_1, \ell_2, \ell_3)]=\C[z_1^{\ell_1}, w_1^{\ell_1}, z_1w_1, z_2^{\ell_2}, w_2^{\ell_2}, z_2w_2, z_3^{\ell_3}, w_3^{\ell_3}, z_3w_3]^{\T_\C^1}/\langle z_1w_1+z_2w_2+z_3w_3\rangle.\]
In particular, we get the following description of $\overline{\calO}^{\min}(\ell_1, \ell_2, \ell_3)$: 
\[\overline{\calO}^{\min}(\ell_1, \ell_2, \ell_3):=\Set{\begin{pmatrix}u_1&x_{12}&x_{13}\\y_{12}&u_2&x_{23}\\y_{13}&y_{23}&u_3\end{pmatrix} \in \gersl_3(\C) \ | \ \text{{\rm All 2 $\times$ 2 minors of}} \begin{pmatrix}u_1^{\ell_1}&x_{12}&x_{13}\\y_{12}&u_2^{\ell_2}&x_{23}\\y_{13}&y_{23}&u_3^{\ell_3}\end{pmatrix}=0}.\]
In the similar way, one can consider the higher dimension analogue $\overline{\calO}^{\min}(\ell_1, \ldots, \ell_s)$ (cf.\ \cite[{Remark 4.9}]{Nag}). 
By the classification \cite[{Theorem 4.8}]{Nag}, it is known that 4-dimensional affine hypertoric varieties are isomorphic to the product of two $A$-type surface singularities or $\overline{\calO}^{\min}(\ell_1, \ell_2, \ell_3)$.

\end{example}


Finally, we remark the following lemma. 

\begin{lemma}{\rm (Gale duality cf. \cite[{Lemma 5.7}]{Nag})}\label{lem:Gale}\\
Assume $A=[\bm{a_1}, \ldots ,\bm{a_n}] \in M_{d \times n}(\R)$ and $B=[\bm{b_1}, \ldots ,\bm{b_n}]^T \in  M_{n \times (n-d)}(\R)$ satisfy the exact sequence 
\[\begin{tikzcd}0\ar[r]&\R^{n-d}\ar[r, "B"]&\R^n\ar[r, "A"]&\R^d\ar[r]&0\end{tikzcd}.\]
For a partition $\{1,2, \ldots ,n\}=I\sqcup J$ and $r \in \Z_{\geq 0}$ satisfying $0 \leq |I|-r \leq n-d, \ 0 \leq r \leq d$, we have 
\[
\dim_\R\Span_\R(\bm{b_i} \ | \ i \in I)=|I|-r \Leftrightarrow \dim_\R\Span_\R(\bm{a_j} \ | \ j \in J)=d-r
.\] 
\end{lemma}

\section{A stratification of affine hypertoric varieties}

In this section, for the later discussion, we recall a result in \cite{PW} on the stratification of $Y_A(0)$ and we describe the singular locus of $Y_A(0)$. 

To describe the singular locus $\Sing (Y_A(0))$ of an affine hypertoric variety $Y_A(0)$, we consider the associated hyperplane arrangement as the following. 

\begin{definition}(The associated hyperplane arrangement $\calH_B^{\alpha}$ to $Y_A(\alpha)$)\label{def:arr}\\
In the above setting, for $\alpha=A\widetilde{\alpha}$ ($\widetilde{\alpha}\in\Z^n$), we define the {associated hyperplane arrangement} $\calH_B^{\alpha}$ to $Y_A(\alpha)$ by 
\[\calH_B^\alpha:=\Set{H_i : \langle \bm{b_i}, -\rangle=-\tilde{\alpha_i}}\subset \R^{n-d}.\]
If $\alpha=0$, we will take $\widetilde{\alpha}$ as $\bm{0}$. 
\end{definition}
\begin{remark}\hspace{2pt}\\
\vspace{-5mm}
\begin{itemize}
\item[(1)] $\calH_B^{\alpha}$ is independent on the choice of $\widetilde{\alpha}$ up to translations. In fact, if $A\widetilde{\alpha}=A\widetilde{\alpha}'$, then $\widetilde{\alpha}'-\widetilde{\alpha}=B\bm{v}\in\Image B$. Thus two arrangements are translated by $\bm{v}$.  
\item[(2)] We consider $\calH_B^{\alpha}$ as a multi-arrangement, i.e., a multiset of hyperplanes. 
\end{itemize}  
\end{remark}

Below, we will mainly consider $\calH_B^0$. A subset $F\subseteq\{1,2,\ldots,n\}$ is called a {\it flat} if it satisfies $F=\{i \ | \ \bigcap_{j\in F}{H_j}\subseteq H_i\}$. For a flat $F$, we set $H_F:=\bigcap_{j\in F}{H_j}$ and define the {\it rank} of $F$ as $\rank F:={\codim} H_F$. In particular, each rank 1 flat is an index set of parallel hyperplanes in $\calH_B^{0}$. 

\noindent For a flat $F$, we consider the following exact sequence induced from\\ $\begin{tikzcd}[column sep=small]0\ar[r]&\Z^{n-d}\ar[r, "B"]&\Z^n\ar[r, "A"]&\Z^d\ar[r]&0\end{tikzcd}$ and the natural projection $\Z^n\twoheadrightarrow\Z^{|F|}$:
 \[
\begin{tikzcd}
0\ar[r]&\Z^{n-d}/H_F\ar[r, "{B}_F"]&\Z^{|F|}\ar[r, "{A}_F"]&\Z^d/\sum_{i \notin F}\Z\bm{a_i}\ar[r]&0
\end{tikzcd}.
\]
In \cite{PW}, for a general $A$, they give a stratification of $Y_A(0)$ by smooth strata which indexed by flats of $\calH_B^0$. Also, they describe the slice of each stratum as the following. Below, we consider $Y_A(0)$ as the set of closed $\T_\C^d$-orbits $[\bm{z}, \bm{w}]:=\T_\C^d\cdot(\bm{z}, \bm{w})$ in $\mu^{-1}(0)$ as usual in geometric invariant theory. 
\begin{theorem}{\rm (\cite[Lemma 2.5]{PW})}\label{thm:PW}\\
In the setting above, 
\[Y_A(0)=\bigsqcup_{F: \text{flat}}{\mathring{Y}(A, 0)^F}\]
is a stratification by $2(n-d-\rank F)$ dimensional smooth strata $\mathring{Y}(A, 0)^F$ defined as the following: 
\[\mathring{Y}(A, 0)^F:={\Set{[\bm{z}, \bm{w}] \in Y_A(0) \ | \ \text{$\T_\C^d\cdot (\bm{z}, \bm{w})$ is closed in $\mu_A^{-1}(0)$, and} \ (z_j, w_j)=0 \Leftrightarrow j \in F}.}\]
Moreover, the slice to each stratum $\mathring{Y}(A, 0)^F$ is given by $Y(A_F, 0)$. 
\end{theorem}

From this theorem, we want to describe the codimension 2 singular locus $\Sigma_{\codim2}$ of $Y_A(0)$, which is the complement of codimension $\geq4$ singular locus in $\Sing (Y_A(0))$. To do so, by interchanging the row vectors $\bm{b_i}$ of $B$, and multiplying $\bm{b_i}$ by $\pm1$ if necessarily, we can assume  
\vspace{2pt}

\noindent {\setlength\arraycolsep{1pt}$B^T\hspace{-5pt}=$\raise1ex\hbox{$\begin{array}{rcccccccl}
&\multicolumn{3}{c}{\smash{\overbrace{\hspace{6ex}}^{\ell_1}}}&&\multicolumn{3}{c}{\smash{\overbrace{\hspace{6ex}}^{\ell_s}}}&\\
\ldelim[{1}{3pt}[]&\bm{b^{(1)}}&\cdots&\bm{b^{(1)}}&\cdots&\bm{b^{(s)}}&\cdots&\bm{b^{(s)}}&\rdelim]{1}{0.1pt}[]\\
\end{array}$} \ . 
Since $B$ is unimodular, if  $k_1\neq k_2$, then $\bm{b^{(k_1)}}\neq\ell\bm{b^{(k_2)}}$ for any $\ell\in\Z\setminus\{0\}$. We denote by $[n]=\bigsqcup_{k=1}^s{F_k}$ the corresponding decomposition of indices set. This means that each rank 1 flat of $\calH_B^0$ is given by $F_k$'s. First, we define a condition that $Y_A(0)$ has no codimension 2 singular locus in terms of the hyperplane arrangement $\calH_B^0$. 

\begin{definition}{\rm (Simple hyperplane arrangements and simple hypertoric varieties)}\label{def:simple}\\
In the above setting, we call $B$ (or $Y_A(0)$) {\it simple} if $\ell_k=1$ for any $k$, that is, the associated hyperplane arrangement $\calH_B^0$ has no multiplicated hyperplanes.   
\end{definition}

Now we can describe the codimension 2 singular locus $\Sigma_{\codim2}$ of $Y_A(0)$ as the following. We will use the second assertion later as well.  
\begin{corollary}\label{cor:codim2sing}
In the setting above, we have the following:
\begin{itemize}
\item[(1)] The codimension 2 singular locus $\Sigma_{\codim2}$ of $Y_A(0)$ is 
\[\Sigma_{\codim2}=\bigsqcup_{k : \ell_k\geq2 }{\mathring{Y}(A, 0)^{F_k}}.\]
Additionally, the slice to $\mathring{Y}(A, 0)^{F_k}$ is the $A_{\ell_k-1}$ type surface singularity $S_{A_{\ell_k-1}}$ (cf.\ Example \ref{ex:Atypesurface}). In particular, $B$ is simple if and only if  
\[\codim_{Y_A(0)} \Sing(Y_A(0))\geq4. \]
\item[(2)] We have 
\begin{align*}
Y_A(0)_{\reg}&\subset Y_A(0)-\overline{\Sigma_{\codim2}}\\
&=\Set{[\bm{z}, \bm{w}] \in Y_A(0) \ | \ {\begin{array}{l}\text{$\T_\C^d\cdot (\bm{z}, \bm{w})$ is closed in $\mu_A^{-1}(0)$, and }\\ \text{$\ell_k\geq2$} \Rightarrow \ z_{j_k}\neq0 \ \text{or} \  w_{j_k}\neq0 \  (\exists j_k \in F_k)\end{array}}}\\
&=\Set{[\bm{z}, \bm{w}] \in Y_A(0) \ | \ {\begin{array}{l}\text{$\T_\C^d\cdot (\bm{z}, \bm{w})$ is closed in $\mu_A^{-1}(0)$, and }\\ \text{$\ell_k\geq2$} \Rightarrow \ \prod_{j\in F_k}{z_j}\neq0 \ \text{or} \  \prod_{j\in F_k}{w_j}\neq0\end{array}}}
\end{align*}
\end{itemize}
\end{corollary}
\begin{proof}
(1) \ See \cite[{Corollary 3.7}]{Nag}. 


(2) \ We only have to show the final equality. Assume that $\T_\C^d\cdot(\bm{z}, \bm{w})\in\mu_A^{-1}(0)$ is closed and we have $z_{j_k}\neq0$ or $w_{j_k}\neq0$ for some $j_k\in F_k$ such that $|F_k|=\ell_k\geq2$. Fix any $k$ such that $\ell_k\geq2$ and we can assume $z_{j_k}\neq0$ without loss of generality. In this setting, we will show $\prod_{j\in F_k}{z_j}\neq0$. If not so, there exists a $j'_k\in F_k$ such that $z_{j'_k}=0$. In particular, we have $z_{j'_k}w_{j'_k}=0$. Note that $(\bm{z}, \bm{w})\in\mu^{-1}(0)$ means $(z_1w_1, \ldots, z_nw_n)^T\in\Image(B^T)$, in particular, we have $0=z_{j'_k}w_{j'_k}=z_{j_k}w_{j_k}$ by $\bm{b_{j_k}}=\bm{b_{j'_k}}=\bm{b^{(k)}}$. Thus, we have $w_{j_k}=0$.   
Now, consider a one parameter subgroup $\rho : \Cstar \to \T_\C^n$ defined by $\rho_{j_k}(s)=s$, $\rho_{j'_k}(s)=s^{-1}$, and $\rho_i(s)=1$ for any $i\in[n]\setminus\{j_k, j'_k\}$. Since we have $\bm{b_{j_k}}=\bm{b_{j'_k}}=\bm{b^{(k)}}$, $\rho$ defines a one parameter subgroup in $\T_\C^d=\Ker (B^T : \T_\C^n \to \T_\C^{n-d})$. Note that the $j_k$-th and $j'_k$-th components of $\rho(s)(\bm{z}, \bm{w})$ is  
\[(sz_{j_k}, s^{-1}w_{j_k}, s^{-1}z_{j'_k}, sw_{j'_k})=(sz_{j_k}, 0, 0, sw_{j'_k})\xrightarrow{s\rightarrow0} (0, 0, 0, 0).\]  
However, this contradicts to the closedness of $\T_\C^d\cdot(\bm{z}, \bm{w})$ and $z_{j_k}\neq0$. 
\end{proof}
By Corollary \ref{cor:codim2sing} (1) and Proposition \ref{prop:codim4fundamental}, we have the following corollary. 

\begin{corollary}\label{cor:hypertoric trivial}
For any simple affine hypertoric variety $Y_A(0)$, the fundamental group $\pi_1(Y_A(0)_{\reg})$ is trivial.  
\end{corollary}
\begin{remark}\hspace{2pt}\\
\vspace{-5mm}
\begin{itemize}
\item[(1)]
Later, we will show that the converse of this corollary is also true (cf. Corollary \ref{cor:converse}).  
\item[(2)]
Although we showed the above corollary by applying  Proposition \ref{prop:codim4fundamental} which is deduced from a very general theorem by Takayama, we can avoid this theorem to prove the above corollary. Actually, by the argument in Proposition \ref{prop:codim4fundamental}, we only have to show $\pi_1(Y_A(\alpha))=0$ for any hypertoric varieties with generic $\alpha$. To show this, we note that any smooth hypertoric variety $Y_A(\alpha)$ is homotopy equivalent to a Lawrence toric variety $X_A(\alpha):=(\C^{2n})^{\alpha-st}/\T_\C^d$ which is a smooth toric variety whose fan is maximal dimension (cf.\ \cite[Corollary 2.7]{HS}). Since it is known that a toric variety whose fan is maximal dimension is simply-connected (cf.\ \cite[Theorem 12.1.10]{CLS}), we deduce $\pi_1(Y_A(\alpha))\cong\pi_1(X_A(\alpha))=0$. 
\end{itemize}
\end{remark}
Although we only need the above corollary for the later discussion essentially, to be complete, we will describe the whole singular locus of $Y_A(0)$ below.   

\begin{corollary}\label{cor:singularlocus}
In the above setting, 
\begin{itemize}
\item[(1)] For any flat $F$, each (or equivalently, some) point in the stratum $\mathring{Y}(A, 0)^F$ is contained in the singular locus $\Sing(Y_A(0))$ if and only if $F$ is {\rm multiplicated}, i.e., $|F|\geq\rank F+1$. In particular, $Y_A(0)$ is smooth if and only if $B\sim I_n$, i.e., $d=0$. In this case, $Y_A(0)= \C^{2n}$. 
\item[(2)] 
We have a stratification: 
\begin{align*}
\Sing(Y_A(0))&=\bigsqcup_{F: \text{flat s.t. $|F|\geq\rank F+1$}}{\mathring{Y}_A(0)^F}\\
&=\Set{[\bm{z}, \bm{w}] \in Y_A(0) \ | \ {\begin{array}{l}\text{$\T_\C^d\cdot (\bm{z}, \bm{w})$ is closed in $\mu_A^{-1}(0)$, and }\\ \text{$\exists$ flat $F$ such that $|F|\geq\rank F+1$ and } \ (z_j, w_j)=0 \  (j \in F)\end{array}}}.
\end{align*}
\end{itemize}
\end{corollary}
\begin{proof}

(1) \ First, we prove that the latter part follows from the former part. Since the origin 0 which is the stratum corresponding to the  flat $[n]$ is contained in the closure of any strata, $Y_A(0)$ is singular if and only if $0\in Y_A(0)$ is a singular point. Then, by applying the former part, we conclude that this is equivalent to $n\geq (n-d)+1$, i.e., $d\geq1$.    

Next, we note that if $F=\emptyset$, then $F$ is not multiplicated by definition, and $\mathring{Y}(A, 0)^F$ is the open smooth stratum in $Y_A(0)$. So we only have to show the claim for nonempty flats. We prove the claim by induction on $\dim Y_A(0)=2(n-d)$. When $2(n-d)=2$, we can assume that $B$ is the $n\times 1$-matrix $(1, 1, \cdots, 1)^T$ (cf.\ Example \ref{ex:Atypesurface}). Then, there is the only one nonempty flat $F=[n]$ corresponding to the stratum $0\in Y_A(0)=S_{A_{n-1}}$. Now, the claim is clear. 

When $2(n-d)\geq4$, for any nonempty flat $F$, we will show that the slice $Y_{A_F}(0)$ to $Y_A(0)^F$ is singular if and only if $F$ is multiplicated. Note that $Y_{A_F}(0)$ is the hypertoric variety associated to the bottom horizontal exact sequence below:    
\[\begin{tikzcd}0\ar[r]&\Z^{n-d}\ar[r, "B"]\ar[d, two heads]&\Z^n\ar[r, "A"]\ar[d, two heads]&\Z^d\ar[r]\ar[d, two heads]&0\\
0\ar[r]&\Z^{n-d}/H_F\ar[r, "{B}_F"]&\Z^{|F|}\ar[r, "{A}_F"]&\Z^d/\sum_{i \notin F}\Z\bm{a_i}\ar[r]&0
\end{tikzcd}.\]
Now, one can easily check that $B_F$ doesn't have any zero row vectors. By definition, we have $\rank F=\codim H_F=\rank (\Z^{n-d}/H_F)$. Then by the induction hypothesis and the latter part of the claim, $Y_{A_F}(0)$ is singular if and only if $|F|\geq \rank (\Z^{n-d}/H_F) +1=\rank F+1$. This completes the proof.  

(2) \ This follows from (1) and Theorem \ref{thm:PW}. 
\end{proof}

\begin{remark}
One can see that $Y_A(0)$ has only isolated singularities if and only if $B$ is equivalent to the one of the following matrices: 
\[\begin{pmatrix}1\\1\\\vdots\\1\end{pmatrix} \ \text{or} \ \begin{pmatrix}1&&0\\&\ddots&\\0&&1\\-1&\cdots&-1\end{pmatrix}.\]
(\underline{Proof} : First, note that if there is a multiplicated flat $F$, then any subflat of $F$ is also a multiplicated flat. Thus, by Corollary \ref{cor:singularlocus} (2), $Y_A(0)$ has only isolated singularities if and only if every flat of $\rank \ (n-d-1)$ is not multiplicated, i.e., any $n-d$ subset $\bm{b_{i_1}}, \ldots, \bm{b_{i_{n-d}}}$ of row vectors of $B$ will give a basis. In terms of matroids, this condition is equivalent to the associated matroid $M(B^T)$ is isomorphic to the uniform matroid $U_{n-d, n}$ (for the basics of matroid theory, see \cite{Ox}). Note that $U_{n-d, n}$ is realized as the associated matroid of a unimodular matrix if and only if $(n-d, n)=(0, n), (1, n), (n-1, n),$ or $(n, n)$ (cf.\ \cite[p.660]{Ox}). When $(n-d, n)=(0, n)$ or $(n, n)$, we can take $B$ as a zero matrix or an identity matrix respectively. Since we don't allow $B$ to contain any zero column vectors, we can ignore the first case. Also, if $B$ is an identity matrix, then $Y_A(0)=\C^{2n}$, in particular, smooth. When $(n-d, n)=(1, n)$ or $(n-1, n)$, we can take $B$ as the claim. In this case, $Y_A(0)= S_{A_{n-1}}$ or $Y_A(0)= \overline{\calO}^{\min}_{A_{n-1}}$ respectively (cf.\ Example \ref{ex:Atypesurface} and \ref{ex:minnilpotent}), and these have an only one singular point $0$.)  
\end{remark}


\section{The universal covers of hypertoric varieties}\label{sec:universal}

In this section, we consider the relation between a given hypertoric variety $Y_A(0)$ and the hypertoric variety $Y_{\underline{A}}(0)$ associated to the simplification $\calH_{\overline{B}}^0$ of the arrangement $\calH_B^0$. More precisely, we will prove that $Y_{\underline{A}}(0)$ gives the universal cover of $Y_A(0)$ and $Y_A(0)$ is a finite group quotient of $Y_{\underline{A}}(0)$ (cf. Theorem \ref{thm:fundamentalgrp}). In the next section, we will give a concrete description of $\pi_1(Y_A(0)_{\reg})$ (cf.\ Proposition \ref{prop:explicit}). 
\vspace{2pt}

First, we can assume {\setlength\arraycolsep{1pt}$B^T\hspace{-5pt}=$\raise1ex\hbox{$\begin{array}{rcccccccl}
&\multicolumn{3}{c}{\smash{\overbrace{\hspace{6ex}}^{\ell_1}}}&&\multicolumn{3}{c}{\smash{\overbrace{\hspace{6ex}}^{\ell_s}}}&\\
\ldelim[{1}{3pt}[]&\bm{b^{(1)}}&\cdots&\bm{b^{(1)}}&\cdots&\bm{b^{(s)}}&\cdots&\bm{b^{(s)}}&\rdelim]{1}{0.1pt}[]\\
\end{array}$} \ , where if $k_1 \neq k_2$, then $\bm{b^{(k_1)}}\neq\pm\bm{b^{(k_2)}}$ (cf.\ Definition \ref{def:operation}). We denote by $[n]=F_1\sqcup\cdots\sqcup F_s$ the corresponding partition of indices of parallel row vectors of $B$. Then, we consider the {\it simplification} $\overline{B}^T:=[\bm{b^{(1)}}, \ldots,  \bm{b^{(s)}}]$} of $B^T$ (cf.\ Definition \ref{def:simple}), and take $\underline{A}=[\bm{\underline{a}_1}, \ldots, \bm{\underline{a}_s}]$ satisfying the exact sequence $\begin{tikzcd}0\ar[r]&\Z^{n-d}\ar[r, "\overline{B}"]&\Z^s\ar[r, "\underline{A}"]&\Z^{d-(n-s)}\ar[r]&0\end{tikzcd}$. Then, we consider two matrices $A_0:=A_{\ell_1-1}\oplus\cdots\oplus A_{\ell_s-1}$ and $B_0:=B_{\ell_1-1}\oplus\cdots\oplus B_{\ell_s-1}$ such that  
\[A_{\ell_i-1}:=\begin{pmatrix} 1& & 0&-1\\ &\ddots&&\vdots\\ 0&&1&-1\\\end{pmatrix}, \ \ \ \ \ \ B_{\ell_i-1}:=\begin{pmatrix}1\\1\\\vdots\\1\end{pmatrix},\]
where $A_{\ell_i-1}$ is a $(\ell_i-1)\times\ell_i$-matrix，and $B_{\ell_i-1}$ is a $\ell_i\times1$-matrix (cf. Example \ref{ex:Atypesurface})．Then, we can construct a unique commutative diagram and the diagram of the corresponding algebraic tori as the following:   

\begin{equation}\tag{*}
\begin{tikzcd}
0\ar[r]&\Z^{n-d}\ar[r, "\overline{B}"]\ar[d, equal]&\Z^{s} \ar[r, "\underline{A}" ]\ar[d, hook, "B_0"]&\Z^{d-(n-s)}\ar[r]\ar[d, hook, "i"]&0\\
0\ar[r]&\Z^{n-d}\ar[r, "B"]\ar[d, two heads]&\Z^n\ar[r, "A"]\ar[d, two heads, "A_0"]&\Z^d\ar[r]\ar[d, two heads, "p"]&0\\
0\ar[r]&0\ar[r]&\Z^{n-s}\ar[r, equal]&\Z^{n-s}\ar[r]&0
\end{tikzcd}, \ \ \ 
\begin{tikzcd}
\T_\C^s&\T_\C^{d-(n-s)}\ar[l, hook', "\underline{A}^T"']\\
\T_\C^n\ar[u, two heads, "B_0^T"']&\T_\C^d\ar[l, hook', "{A}^T"']\ar[u, two heads, "i^*"']\\
\T_\C^{n-s}\ar[u, hook', "A_0^T"']\ar[r, equal]&\T_\C^{n-s}\ar[u, hook', "p^*"']
\end{tikzcd}.
\end{equation}
In particular, $i(\bm{\underline{a}_k})=\sum_{j\in F_k}{\bm{a_j}}$ will hold. We called the affine hypertoric variety $Y_{\underline{A}}(0)$ the {\it simplification} of $Y_A(0)$ (cf.\ Definition \ref{def:simple}). 
To state the result, we consider the following embeddings: 
\[\begin{tikzcd}\Gamma:=\prod_{k=1}^s{\Z/\ell_k\Z}\ar[r, hook]& \T_\C^s& \T_\C^{d-(n-s)} \ar[l, hook', "\underline{A}^T"'] \end{tikzcd},\]
where we consider $\Gamma$ as a multiplicative subgroup of $\T_\C^s$ in the usual way. 
Then, the next theorem is the main result in this section.

\begin{theorem}{\rm (The universal cover is given by the simplification)}\label{thm:fundamentalgrp}\\
In the above setting, for $Y_A(0)=\C^{2n}/\hspace{-3pt}/\hspace{-3pt}/\T_\C^d$ and its simplification $Y_{\underline{A}}(0)=\C^{2s}/\hspace{-3pt}/\hspace{-3pt}/\T_\C^{d-(n-s)}$, we have the following {$\Cstar$-equivariant} commutative diagram:  
\[\begin{tikzcd}
Y_{\underline{A}}(0)\ar[r, "\vphi"]\ar[d]&Y_A(0)\\
Y_{\underline{A}}(0)/(\Gamma/\Gamma\cap\T_\C^{d-(n-s)})\ar[ur, "\cong"']
\end{tikzcd}.\]

Moreover, {$\vphi$ preserves symplectic structures,} 
and $\Gamma/\Gamma\cap\T_\C^{d-(n-s)}$ acts on $\vphi^{-1}(Y_A(0)_{\reg})$ freely. In particular, we have 
\[\pi_1(Y_A(0)_{\reg})\cong\Gamma/\Gamma\cap\T_\C^{d-(n-s)}.\]
\end{theorem}
This theorem gives complete answers for Problem \ref{problem:fundamental} and Problem \ref{problem:universalcover}. Moreover, this theorem says that taking the universal cover of $Y_A(0)$ in the geometric side corresponds to taking the simplification of the hyperplane arrangement $\calH_B^0$ in the combinatorial side as the following. 
\begin{figure}[h]
\begin{center}
\includegraphics[width=200pt, height=50pt]{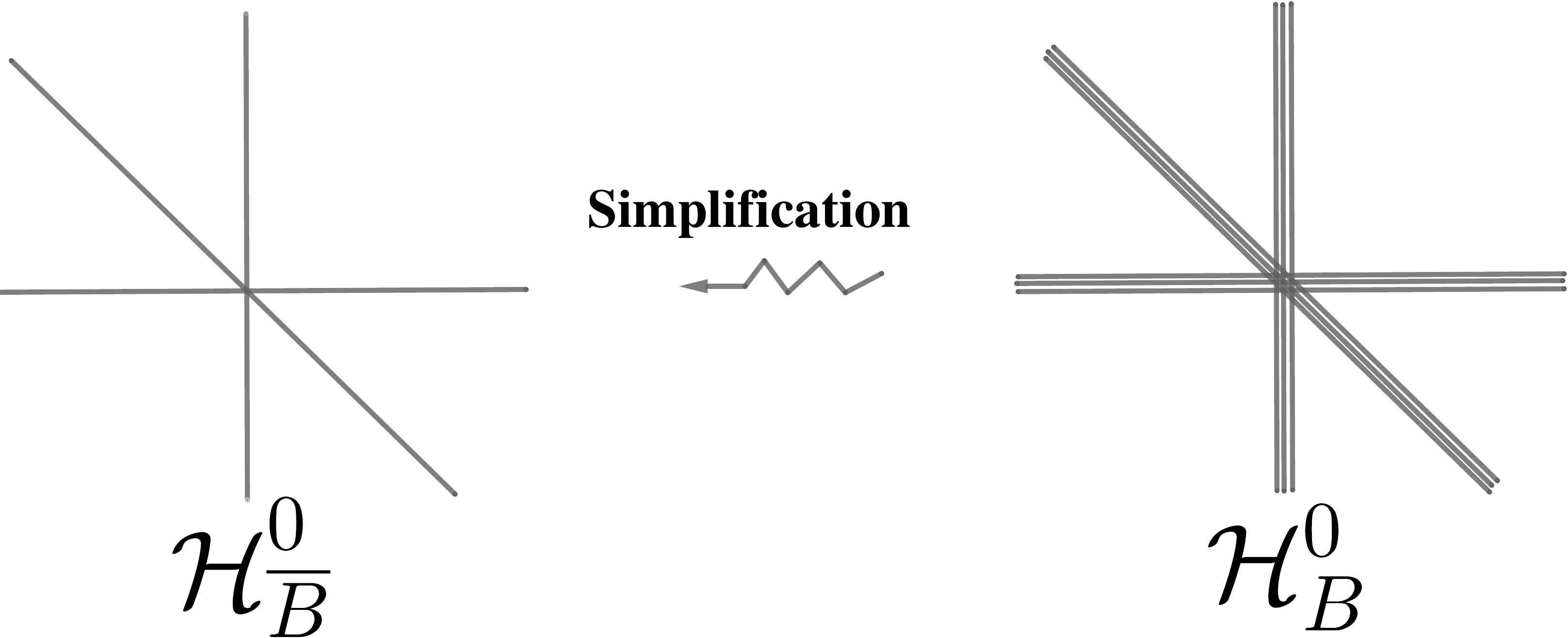}
\end{center}
\end{figure}

\begin{remark}\label{rem:universal}\hspace{2pt}\\
\vspace{-5mm}
\begin{itemize}
\item[(1)] 
In Definition \& Proposition \ref{def and prop:universal cover}, we replace the original $\Cstar$-action $\sigma_Y$ by $\sigma_Y^m$ to lift this action to the universal cover $\overline{Y}$. However, in the hypertoric setting, the $\Cstar$-action on $Y_A(0)$ already lift to $Y_{\underline{A}}(0)$, so we don't have to replace by a new $\Cstar$-action. 
\item[(2)] In \cite[{Theorem 4.2}]{Nag}, we showed that the isomorphism classes of affine hypertoric varieties $Y_A(0)$ as conical symplectic varieties are bijective to the isomorphism classes of the associated matroids $M(B^T)$ (or $M(A)$). In particular, as conical symplectic varieties, 
\[Y_A(0)\cong Y_{A'}(0) \ \Rightarrow \ Y_{\underline{A}}(0)\cong Y_{\underline{A'}}(0)\]
does hold since $M(\overline{B}^T)$ is isomorphic to the simplification of $M(B^T)$ which is determined uniquely up to isomorphisms (cf.\ \cite[{Corollary 4.5}]{Nag}). From this viewpoint, the above theorem can be seen as a geometric interpretation of this combinatorial implication. 
\item[(3)]
In general, not necessarily every finite covering of $Y_A(0)$ is given by another hypertoric variety. For example, if $Y_A(0)=S_{A_1}\times S_{A_1}=\C^2\times\C^2/(\Z_2\times\Z_2)$, then it is well-known that the degree 2 cover $\C^4/\langle 1, 1\rangle$ doesn't admit any symplectic resolution. Thus, this cannot be any hypertoric variety. 
\end{itemize}
\end{remark}

\textbf{The idea of the proof.} $Y_A(0)$ is defined as the symplectic reduction $\C^{2n}/\hspace{-3pt}/\hspace{-3pt}/\T_\C^d$ of $\C^{2n}$ by a torus embedding $\T_\C^d\overset{A^T}{\hookrightarrow}\T_\C^n$. On the other hand, one can take a ``symplectic reduction by two stages'' (for the detail, see Definition \& Lemma \ref{def:twostages}) as the following: 
\[Y_A(0)=\C^{2n}/\hspace{-3pt}/\hspace{-3pt}/\T_\C^d\cong\left(\C^{2n}/\hspace{-3pt}/\hspace{-3pt}/\T_\C^{n-s}\right)/\hspace{-3pt}/\hspace{-3pt}/\T_\C^{d-(n-s)},\]
where $\T_\C^{d-(n-s)} \hspace{-2pt} :=\T_\C^d/\T_\C^{n-s}$, and the action $\T_\C^{n-s} \hspace{-2pt} = \hspace{-2pt} \prod_{k=1}^{s}{\T_\C^{\ell_k-1}}\curvearrowright\C^{2n} \hspace{-2pt} = \hspace{-2pt} \prod_{k=1}^s{\C^{2\ell_k}}$ is the same as the product of the one of Example \ref{ex:Atypesurface} ($A$-type surface singularity). Next, one can show that
\[\C^{2n}/\hspace{-3pt}/\hspace{-3pt}/\T_\C^{n-s}\cong\C^{2s}/\Gamma\]
is a ``torus'' equivariant isomorphism as conical symplectic varieties with respective to the torus identification  $\T_\C^{d-(n-s)}\cong \T_\C^{d-(n-s)}/\Gamma\cap \T_\C^{d-(n-s)}=:T_\Gamma$ which is induced from 
\[\begin{tikzcd}
\T_\C^s\ar[d, two heads, "\psi_\ell"]&\T_\C^{d-(n-s)}\ar[l, hook', "\underline{A}^T"']\ar[r, two heads]\ar[d, two heads, "\psi_\ell"]&T_\Gamma:=\T_\C^{d-(n-s)}/\Gamma\cap\T_\C^{d-(n-s)}\ar[dl, "\sim"]\\
\T_\C^s&\T_\C^{d-(n-s)}\ar[l, hook', "\underline{A}^T"']
\end{tikzcd},\]
where $\psi_\ell : \T_\C^s \to \T_\C^s$ is defined as $(t_1, \ldots, t_s)\mapsto (t_1^{\ell_1}, \ldots, t_s^{\ell_s})$. 
Hence, we have natural isomorphisms
\[Y_A(0) \cong \left(\C^{2n}/\hspace{-3pt}/\hspace{-3pt}/\T_\C^{n-s}\right)/\hspace{-3pt}/\hspace{-3pt}/\T_\C^{d-(n-s)} \cong\left(\C^{2s}/\Gamma\right)/\hspace{-3pt}/\hspace{-3pt}/\left(\T_\C^{d-(n-s)}/\Gamma\cap \T_\C^{d-(n-s)}\right)=\left(\C^{2s}/\Gamma\right)/\hspace{-3pt}/\hspace{-3pt}/T_\Gamma.\]
Finally, we can show the following natural isomorphism which also can be interpreted as the special case of a symplectic reduction by two stages: 
\[\left(\C^{2s}/\Gamma\right)/\hspace{-3pt}/\hspace{-3pt}/T_\Gamma\cong \left(\C^{2s}/\hspace{-3pt}/\hspace{-3pt}/\T_\C^{d-(n-s)}\right)\bigg/\left(\Gamma/\Gamma\cap \T_\C^{d-(n-s)}\right)=Y_{\overline{A}}(0)/\left(\Gamma/\Gamma\cap \T_\C^{d-(n-s)}\right).\] 
By combining these isomorphisms, we obtain the desired isomorphism. 

\begin{example}(The case of $\overline{\calO}^{\min}(\ell_1, \ell_2, \ell_3)$ cf. Example \ref{ex:Ominlll})\\
We demonstrate the above idea in the case of $\overline{\calO}^{\min}(\ell_1, \ell_2, \ell_3)$. First, we note that $\overline{\calO}^{\min}_{A_2}=\C^{2}\times\C^{2}\times\C^{2}/\hspace{-3pt}/\hspace{-3pt}/\T_\C^{1}$, and $\Gamma=\Z_{\ell_1}\times\Z_{\ell_2}\times\Z_{\ell_3}$.  Then, we have the following isomorphisms. 
\begin{align*}
\overline{\calO}^{\min}(\ell_1, \ell_2, \ell_3)&= \C^{2\ell_1}\times\C^{2\ell_2}\times\C^{2\ell_3}/\hspace{-3pt}/\hspace{-3pt}/\T_\C^{\ell_1+\ell_2+\ell_3-2}\\
&\cong\left(\C^{2\ell_1}\times\C^{2\ell_2}\times\C^{2\ell_3}/\hspace{-3pt}/\hspace{-3pt}/\T_\C^{\ell_1-1}\times\T_\C^{\ell_2-1}\times\T_\C^{\ell_3-1}\right)/\hspace{-3pt}/\hspace{-3pt}/\T_\C^{1}\\
&\cong \left(\C^2\times \C^2\times \C^2\right/\Gamma)/\hspace{-3pt}/\hspace{-3pt}/\left(\T_\C^{1}/\T_\C^1\cap \Gamma\right)\\
&\cong \left(\C^{2}\times\C^{2}\times\C^{2}/\hspace{-3pt}/\hspace{-3pt}/\T_\C^{1} \right)/\left(\Gamma/\Gamma\cap\T_\C^1 \right)\\
&\cong \overline{\calO}^{\min}_{A_2}/\left(\Gamma/\Gamma\cap\T_\C^1 \right),
\end{align*}
where $\begin{tikzcd}\Gamma\ar[r, hook]& \T_\C^3& \T_\C^{1} \ar[l, hook', "{\begin{pmatrix}t\\ t\\ t\\ \end{pmatrix}}"'] \end{tikzcd}$. We can actually see these isomorphisms through the following natural embedding of coordinate rings. 
\[\begin{tikzcd}
\C[\overline{\calO}^{\min}(\ell_1, \ell_2, \ell_3)]=\C[z_1^{\ell_1}, w_1^{\ell_1}, z_1w_1, z_2^{\ell_2}, w_2^{\ell_2}, z_2w_2, z_3^{\ell_3}, w_3^{\ell_3}, z_3w_3]^{\T_\C^1}/\langle z_1w_1+z_2w_2+z_3w_3\rangle\ar[d, hook]\\
\C[\overline{\calO}^{\min}_{A_2}]=\C[z_1, w_1, z_2, w_2, z_3, w_3]^{\T_\C^1}/\langle z_1w_1+z_2w_2+z_3w_3\rangle
\end{tikzcd},\]
where $\T_\C^1$-action is $z_i\mapsto tz_i$, and $w_i \mapsto t^{-1}w_i$. 
\end{example}

Now, we give a more precise statement of the above idea and observation. 
\begin{proposition}\label{prop:bystage}
In the above setting, there exists the following natural commutative diagram of conical symplectic varieties, and $f_i \ (i=1, 2, 3)$ are isomorphisms as conical symplectic varieties.   
\[\begin{tikzcd}
Y_{\underline{A}}(0)=\C^{2s}/\hspace{-3pt}/\hspace{-3pt}/\T_\C^{d-(n-s)}\ar[dr]\ar[d]&&&\\
Y_{\underline{A}}(0)/(\Gamma/\Gamma\cap\T_\C^{d-(n-s)})\ar[r, "f_1", "\sim"']&(\C^{2s}/\Gamma)/\hspace{-3pt}/\hspace{-3pt}/T_\Gamma\ar[r ,"f_2", "\sim"']&(\C^{2n}/\hspace{-3pt}/\hspace{-3pt}/\T_\C^{n-s})/\hspace{-3pt}/\hspace{-3pt}/\T_\C^{d-(n-s)}\ar[r, "f_3", "\sim"']&Y_A(0):=\C^{2n}/\hspace{-3pt}/\hspace{-3pt}/\T_\C^d
\end{tikzcd},
\]
where $T_\Gamma:=\T_\C^{d-(n-s)}/\Gamma\cap\T_\C^{d-(n-s)}$. 
\end{proposition}
To prove this, we should give a precise definition of the symplectic reduction by two stages. First, we fix a vector space $V$ and a maximal torus $T\subset GL(V)$, then we define a symplectic reduction of $V$ by a closed subgroup $H\subset T$ in general as follows. We define a moment map as the following:  
\[\mu_H : T^*V \to \gerh^*:=\Lie(H)^* : (\bm{z}, \bm{w}) \mapsto (a \mapsto \bm{w}(a\cdot \bm{z})).\]
Since $\mu_H$ is $T$-invariant, in particular $H$-invariant, we define the symplectic reduction $V/\hspace{-3pt}/\hspace{-3pt}/H$ by $\mu_H^{-1}(0)/\hspace{-3pt}/H:=\Spec\C[\mu_H^{-1}(0)]^H$. This has a natural Poisson structure induced from the standard one of $T^*V$. 
\begin{remark}
By definition, for any linear representation of a finite group $\Gamma\subset GL(V)$, the symplectic reduction is the same as the usual quotient, i.e., $T^*V/\hspace{-3pt}/\hspace{-3pt}/\Gamma=T^*V/\Gamma$.  
\end{remark}
Next, we define the symplectic reduction by two stages as the following. In real symplectic case, such operation is well studied (cf.\ \cite{MMOPR}), but we cannot find a reference for holomorphic (algebraic) symplectic case. 
\begin{definition and lemma}{\rm (Symplectic reduction by two stages in GIT setting)}\label{def:twostages}
In the above setting, we consider a closed subgroup $K$ of $H\subset T\subset GL(V)$. Then, the moment map $\mu_H$ induces a $H/K$-invariant morphism $\Phi_{H/K}: T^*V/\hspace{-3pt}/K \to \gerh^*$, and $\Phi_{H/K}$ induces a $H/K$-invariant morphism 
\[\mu_{H/K} : \mu_K^{-1}(0)/\hspace{-3pt}/K \to (\gerh/\gerk)^*\]
such that the following is commutative:  
\begin{equation*}\tag{*}
\begin{tikzcd}
T^*V/\hspace{-3pt}/K\ar[d, "\Phi_{H/K}"]&\mu_K^{-1}(0)/\hspace{-3pt}/K\ar[l, hook']\ar[d, "\mu_{H/K}"]\\
\gerh^*&(\gerh/\gerk)^*\ar[l, hook']
\end{tikzcd}.\end{equation*}
Then, we define $(T^*V/\hspace{-3pt}/\hspace{-3pt}/K)/\hspace{-3pt}/\hspace{-3pt}/(H/K):=\mu_{H/K}^{-1}(0)/\hspace{-3pt}/(H/K)$, and we call this a {\rm symplectic reduction by two stages} (with respect to $H$ and $K$). In this case, we have a natural $\Cstar\times (T/H)$-equivariant isomorphism as Poisson varieties: 
\[T^*V/\hspace{-3pt}/\hspace{-3pt}/H\cong (T^*V/\hspace{-3pt}/\hspace{-3pt}/K)/\hspace{-3pt}/\hspace{-3pt}/(H/K).\]

\end{definition and lemma}
\begin{proof}
Since $\mu_H : T^*V \to \gerh^*$ is $K$-invariant, $\mu_H$ induces a $H/K$-invariant morphism $\Phi_{H/K} : T^*V/\hspace{-3pt}/\hspace{-3pt}/K \to \gerh^*$. Moreover, by definition, we have the following commutative diagram:  
\[\begin{tikzcd}
&T^*V\ar[dl, "\mu_K"]\ar[d, two heads, "\mu_H"]\\
\gerk^*&\gerh^*\ar[l, two heads, "i^*"]
\end{tikzcd}.\]
In particular, $\mu_H$ induces a morphism $\mu_K^{-1}(0) \to (\gerh/\gerk)^*$, and this is $H/K$-invariant. Thus, we obtain a morphism $\mu_{H/K} : \mu_K^{-1}(0)/\hspace{-3pt}/K \to (\gerh/\gerk)^*$ satisfying the diagram (*). 
Since we have a natural isomorphism $T^*V/\hspace{-3pt}/H \cong (T^*V/\hspace{-3pt}/K)/\hspace{-3pt}/(H/K)$ and (*) is commutative, we obtain the following diagram and $(T^*V/\hspace{-3pt}/\hspace{-3pt}/K)/\hspace{-3pt}/\hspace{-3pt}/(H/K):=\mu_{H/K}^{-1}(0)/\hspace{-3pt}/(H/K)$ is identified with the central fiber of $\overline{\mu}_H$, and $ \overline{\Phi}_{H/K}, \overline{\mu}_{H/K}$:  
\[\begin{tikzcd}
&T^*V/\hspace{-2pt}/H\ar[rr, "\sim"]\ar[ddl, two heads, "\overline{\mu}_H" near start]&&\left(T^*V/\hspace{-2pt}/K\right)/\hspace{-2pt}/(H/K)\ar[ddl, two heads, "\overline{\Phi}_{H/K}" near end]&(\mu_K^{-1}(0)/\hspace{-2pt}/K)/\hspace{-2pt}/(H/K)\ar[l, hook']\ar[ddl, two heads, "\overline{\mu}_{H/K}"]\\
T^*V\ar[rr, crossing over, two heads]\ar[ur, two heads]\ar[d, "\mu_H"']&&T^*V/\hspace{-2pt}/K\ar[ur, two heads]\ar[d, "\Phi_{H/K}"']&\mu_K^{-1}(0)/\hspace{-2pt}/K\ar[l, crossing over, hook']\ar[ur, two heads]\ar[d, "\mu_{H/K}"]&\\
\gerh^*&&\gerh^*\ar[ll, equal]&(\gerh/\gerk)^*\ar[l, hook']&
\end{tikzcd},\]
where $\overline{\mu}_H$, and $ \overline{\Phi}_{H/K}, \overline{\mu}_{H/K}$ are the induced morphism from $\mu_H$, and $\Phi_{H/K}$, $\mu_{H/K}$ respectively. This gives an isomorphism $T^*V/\hspace{-3pt}/\hspace{-3pt}/H\cong (T^*V/\hspace{-3pt}/\hspace{-3pt}/K)/\hspace{-3pt}/\hspace{-3pt}/(H/K)$. Also, one can easily check that this is $\Cstar$-equivariant and preserves the Poisson structures. 
\end{proof}

\begin{corollary}\label{cor:naturalHam}
Let $T\subset GL(V)$ be a maximal torus and consider the natural Hamiltonian action $T\curvearrowright T^*V$. Then for any closed (possibly non-connected) algebraic subgroup $K, H\subset T$, there exists natural $\Cstar\times (T/\langle H, K\rangle)$-equivariant isomorphisms as Poisson varieties: 
\[T^*V/\hspace{-3pt}/\hspace{-3pt}/\langle H, K\rangle\cong\left(T^*V/\hspace{-3pt}/\hspace{-3pt}/H\right)/\hspace{-3pt}/\hspace{-3pt}/(K/K\cap H)\cong \left(T^*V/\hspace{-3pt}/\hspace{-3pt}/K\right)/\hspace{-3pt}/\hspace{-3pt}/(H/H\cap K),\]
where $\langle H, K\rangle\subset T$ is the subgroup generated by $H$ and $K$．
\end{corollary}
\begin{proof}
Since $\langle H, K\rangle/H\cong K/K\cap H$ and $\langle H, K\rangle/K \cong H/H\cap K$, we obtain the desired isomorphism from the above Definition \& Lemma. 
\end{proof}

\begin{proof}[\textbf{Proof of Proposition \ref{prop:bystage}}]
To prove that $f_1$ (resp. $f_3$) is an isomorphism, we apply the above corollary as $T=\T_\C^s$, $K=\Gamma$, and $H=\T_\C^{d-(n-s)}$ (resp. $T=\T_\C^n$, $K=\T_\C^d$, and $H=\T_\C^{n-s}$). 
Thus, we only have to show that $f_2$ is an $\Cstar$-equivariant isomorphism as Poisson varieties. To prove this, we show that the following natural vertical isomorphism as conical symplectic varieties (cf. Example \ref{ex:Atypesurface})
\begin{equation*}\tag{*}
\begin{tikzcd}
T_\Gamma \ar[r, hook, "\underline{A}^T"]&\T_\C^s/\Gamma\ar[r, bend left]&\C^{2s}/\Gamma\ar[d, "\wr"]\\
\T_\C^{d-(n-s)}=\T_\C^d/\T_\C^{n-s}\ar[r, hook, "A^T"]&\T_\C^n/\T_\C^{n-s}\ar[r, bend left]&\C^{2n}/\hspace{-3pt}/\hspace{-3pt}/\T_\C^{n-s}
\end{tikzcd}
\end{equation*}
is equivariant with respect to the natural identification $T_\Gamma \xrightarrow{\sim} \T_\C^{d-(n-s)}$ induced from: 
\[\begin{tikzcd}
\T_\C^s\ar[d, two heads, "\psi_\ell"]&\T_\C^{d-(n-s)}\ar[l, hook', "\underline{A}^T"']\ar[r, two heads]\ar[d, two heads, "\psi_\ell"]&T_\Gamma:=\T_\C^{d-(n-s)}/\Gamma\cap\T_\C^{d-(n-s)}\ar[dl, "\sim"]\\
\T_\C^s&\T_\C^{d-(n-s)}\ar[l, hook', "\underline{A}^T"']
\end{tikzcd},\]
where $\psi_\ell : \T_\C^s \to \T_\C^s$ is defined as $(t_1, \ldots, t_s)\mapsto (t_1^{\ell_1}, \ldots, t_s^{\ell_s})$. 
More precisely, we should show that the center positioned subdiagram of the following diagram is commutative:  
\[\begin{tikzcd}[column sep=small, contains/.style = {draw=none,"\in" description,sloped}, icontains/.style = {draw=none,"\ni" description,sloped}, inclusion/.style = {draw=none,"\subset" description,sloped}, niclusion/.style = {draw=none,"\supset" description,sloped}]
                    &\T_\C^{d-(n-s)}\ar[r, hook, "\underline{A}^T"]\ar[d, two heads]&\T_\C^s\ar[r, bend left]\ar[d, two heads]&\C^{2s}\ar[d, two heads]\\
                    &T_\Gamma:=\T_\C^{d-(n-s)}/\Gamma\cap\T_\C^{d-(n-s)}\ar[r, hook, "\underline{A}^T"]\ar[dd, "\wr"]&\T_\C^s/\Gamma\ar[r, bend left]&\C^{2s}/\Gamma\ar[dd, "\wr"]\\
\T_\C^{d-(n-s)}\ar[uur, bend left, equal]\ar[ur, two heads]\ar[dr, two heads, "\psi_\ell"]&&&\\
                    &\T_\C^d/\T_\C^{n-s}=\T_\C^{d-(n-s)}\ar[r, hook, "\underline{A}^T"]&\T_\C^s=\T_\C^n/\T_\C^{n-s}\ar[r, bend left]&\C^{2n}/\hspace{-3pt}/\hspace{-3pt}/\T_\C^{n-s}\\
                    &\T_\C^d\ar[u, two heads, "i^*"]\ar[r, hook, "{A}^T"]&\T_\C^n\ar[r, bend left]\ar[u, two heads, "{B_0}^T"]&\C^{2n}
\end{tikzcd}\]
To do so, we describe how an element $\bm{r}\in\T_\C^{d-(n-s)}$ will go through in the above diagram, and we express the coordinate rings of the varieties and isomorphisms among them below (cf.\ Example \ref{ex:Atypesurface}), where we take a $\bm{t}\in\T_\C^d$ such that $i^*(\bm{t})=\psi_\ell(\bm{r})$. Note that $\C^{2n}/\hspace{-3pt}/\hspace{-3pt}/\T_\C^{n-s}=\prod_{k=1}^s{\C^{2\ell_k}/\hspace{-3pt}/\hspace{-3pt}/\T_\C^{\ell_k-1}}$ and $\C^{2s}/\Gamma=\prod_{k=1}^s{\C^2/\Z_{\ell_k}}$. 

\begin{equation*}
\begin{tikzcd}[column sep=small, contains/.style = {draw=none,"\in" description,sloped}, icontains/.style = {draw=none,"\ni" description,sloped}, inclusion/.style = {draw=none,"\subset" description,sloped}, niclusion/.style = {draw=none,"\supset" description,sloped}]
						&\bm{r}\ar[r, mapsto, "\underline{A}^T"]\ar[d, mapsto]&(\bm{r}^{\bm{\underline{a}_1}}, \ldots, \bm{r}^{\bm{\underline{a}_s}})\ar[d, mapsto]\ar[r, bend left]&\C[\C^{2s}]\ar[r, equal]&\bigotimes_{k=1}^s{\C[\underline{z}_k, \underline{w}_k]}\\
                    &\overline{\bm{r}}\ar[r, mapsto, "\underline{A}^T"]\ar[dd, mapsto, "\wr"]&(\overline{\bm{r}}^{\bm{\underline{a}_1}}, \ldots, \overline{\bm{r}}^{\bm{\underline{a}_s}})\ar[r, bend left]&\C[\C^{2s}/\Gamma]\ar[r, equal]&\bigotimes_{k=1}^s{\C[u_k, x_k, y_k]/\langle u_k^{\ell_k}-x_ky_k\rangle}\ar[u, hook, "{u_k=\underline{z}_k\underline{w}_k, \ \ x_k=\underline{z}_k^{\ell_k}, \ \ y_k=\underline{w}_k^{\ell_k}}"']\\
{\bm{r}}\ar[ur, mapsto]\ar[uur, equal, bend left]\ar[dr, mapsto, "\psi_\ell"]&&&\\
                    &\psi_{\ell}({\bm{r}})\ar[r, mapsto, "\underline{A}^T"]&{\begin{array}{c}({\bm{r}}^{\ell_1\bm{\underline{a}_1}}, \ldots, {\bm{r}}^{\ell_s\bm{\underline{a}_s}})\\ \shortparallel \\ (\bm{t}^{\sum_{i\in F_1}{\bm{a_i}}}, \ldots, \bm{t}^{\sum_{i\in F_s}{\bm{a_i}}})\end{array}}\ar[r, bend left, shift right]&\C[\C^{2n}/\hspace{-3pt}/\hspace{-3pt}/\T_\C^{n-s}]\ar[r, equal]&\bigotimes_{k=1}^s{\fr<{\C[\{z_{i}w_{i}\}_{i\in F_k}, \prod_{i\in F_k}{z_{i}}, \prod_{i\in F_k}w_{i}]}/{\langle z_{i}w_{i}=z_{j}w_{j}\rangle_{i, j\in F_k}}>}\ar[uu, "{\begin{array}{l}\wr \  u_k=\ell_k z_{i}w_{i} \ (i\in F_k), \\ \ \ \ x_k=\prod_{i\in F_k}{\sqrt{\ell_k}z_{i}}, \\ \ \ \ y_k=\prod_{i\in F_k}{\sqrt{\ell_k}w_{i}}\end{array}}"']\\
&{\bm{t}}\ar[u, mapsto, "i^*"]\ar[r, mapsto, "{A}^T"]&(\bm{t}^{\bm{a_1}}, \ldots, \ldots, \bm{t}^{\bm{a_n}})\ar[u, mapsto, "{B_0}^T"]\ar[r, bend left]&\C[\C^{2n}]\ar[r, equal]&\bigotimes_{k=1}^s{\C[\{z_{i}, w_{i}\}_{i\in F_k}]}
\end{tikzcd}\end{equation*}
Since $i(\bm{\underline{a}_k})=\sum_{j\in F_k}{\bm{a_j}}$ by definition, we can easily check the equivariance of (*) from the above diagram.

Next, we show that $(\C^{2s}/\Gamma)/\hspace{-3pt}/\hspace{-3pt}/T_\Gamma\cong(\C^{2n}/\hspace{-3pt}/\hspace{-3pt}/\T_\C^{n-s})/\hspace{-3pt}/\hspace{-3pt}/\T_\C^{d-(n-s)}$. By definition, we only have to show the commutativity of the middle square in the diagram below, that is, the compatibility of the moment maps: 
\[\begin{tikzcd}
\C^{2n}\ar[d, two heads, "\mu_A"]\ar[r, two heads]&\C^{2n}/\hspace{-3pt}/\T_\C^{n-s}&\C^{2n}/\hspace{-3pt}/\hspace{-3pt}/\T_\C^{n-s}\ar[d, two heads, "\mu_{\T_\C^d/\T_\C^{n-s}}"]\ar[l, hook']\ar[r, "\sim"]&\C^{2s}/\Gamma\ar[d, "\mu_{T_\Gamma}"]&\C^{2s}\ar[l, two heads]\ar[dd, two heads, "\mu_{\overline{A}}"]\\
(\gert_\C^d)^*&&(\gert_\C^d/\gert_\C^{n-s})^*\ar[ll, hook']\ar[r, "\sim"]\ar[drr, hook, "\psi_\ell^*"']&(\gert_\Gamma)^*\ar[dr, hook]&\\
&&&&(\gert_\C^{d-(n-s)})^*
\end{tikzcd}.\]
This can be easily checked. Finally, we note that $\vphi$ preserves the standard symplectic structure. In fact, $f_1$ and $f_3$ are naturally isomorphic as Poisson varieties by Corollary \ref{cor:naturalHam}, and $f_2$ also preserves the symplectic structure since the isomorphism $\C^{2n}/\hspace{-3pt}/\hspace{-3pt}/\T_\C^{n-s}\cong\C^{2s}/\Gamma$ preserves symplectic structures as we have seen in Example \ref{ex:Atypesurface}.    

\end{proof}

\begin{proof}[\textbf{Proof of Theorem \ref{thm:fundamentalgrp}}]
From Proposition \ref{prop:bystage}, we get a finite morphism $\vphi : Y_{\underline{A}}(0) \to Y_A(0)$ which preserves the symplectic structures. So, we only have to show that $\Gamma/\Gamma\cap\T_\C^{d-(n-s)}$ acts on $\vphi^{-1}(Y_A(0)_{\reg})$ freely. As seen in the proof of Proposition \ref{prop:bystage}, $\vphi$ is induced from the composition $\C^{2s} \to \C^{2s}/\Gamma \xrightarrow{\sim} \C^{2n}/\hspace{-3pt}/\hspace{-3pt}/\T_\C^{n-s}$ which corresponds to the following homomorphism of coordinate rings: 
\[\C[\C^{2s}]=\bigotimes_{k=1}^s{\C[\underline{z}_k, \underline{w}_k]}\leftarrow \C[\C^{2n}/\hspace{-3pt}/\hspace{-3pt}/\T_\C^{n-s}]=\bigotimes_{k=1}^s{\fr<{\C[\{z_{i}w_{i}\}_{i\in F_k}, \prod_{i\in F_k}{z_{i}}, \prod_{i\in F_k}w_{i}]}/{\langle z_{i}w_{i}=z_{j}w_{j}\rangle_{i, j\in F_k}}>},\]
where $\ell_k z_iw_i \mapsto \underline{z}_k\underline{w}_k \ (i\in F_k)$, $\prod_{i\in F_k}{\sqrt{\ell_k}z_i} \mapsto \underline{z}_k^{\ell_k}$, $\prod_{i\in F_k}{\sqrt{\ell_k}w_i} \mapsto \underline{w}_k^{\ell_k}$. 
In particular, if $\vphi([\bm{\underline{z}}, \bm{\underline{w}}])=[\bm{z}, \bm{w}]\in Y_A(0)$ satisfies $\prod_{i\in F_k}{z_{i}}\neq0$ or $\prod_{i\in F_k}{w_{i}}\neq0$ for some $k$, then $\underline{z}_k\neq0$ or $\underline{w}_k\neq0$. Thus, by Corollary \ref{cor:codim2sing} (2), we have 
\[\vphi^{-1}(Y_A(0)_{\reg})\subset \Set{[\bm{\underline{z}}, \bm{\underline{w}}] \in Y_{\underline{A}}(0) \ | \ {\begin{array}{l}\text{$\T_\C^{d-(n-s)}\cdot (\bm{\underline{z}}, \bm{\underline{w}})$ is closed in $\mu_{\underline{A}}^{-1}(0)$, and }\\ \text{$\ell_k\geq2$} \Rightarrow \ \underline{z}_k\neq0 \ \text{or} \  \underline{w}_k\neq0\end{array}}}\]
This implies that  $\Gamma/\Gamma\cap\T_\C^{d-(n-s)}$ acts on $\vphi^{-1}(Y_A(0)_{\reg})$ freely. 

\end{proof}

\begin{remark}\label{rem:phicodim}
Note that $\codim_{Y_{\underline{A}}(0)}{(Y_{\underline{A}}(0)-\vphi^{-1}(Y_A(0)_{\reg}))}\geq2$, in particular, $\codim_{Y_{\underline{A}}(0)_{\reg}}{(Y_{\underline{A}}(0)_{\reg}-\vphi^{-1}(Y_A(0)_{\reg}))}\geq2$. In fact, we only have to show $\codim_{Y_A(0)}\Sing(Y_A(0))\geq2$ since $\vphi$ is a finite surjective morphism. However, this follows from Theorem \ref{thm:PW}. We will use this property in the proof of Lemma \ref{lem:reduction to simple}. 
\end{remark}

\section{A computation of the fundamental group}
In the previous section, we described the fundamental group $\pi_1(Y_A(0)_{\reg})\cong \Gamma/\Gamma\cap\T_\C^{d-(n-s)}$, where the intersection is considered as  
\[\begin{tikzcd}\Gamma:=\prod_{k=1}^s{\Z/\ell_k\Z}\ar[r, hook]& \T_\C^s& \T_\C^{d-(n-s)} \ar[l, hook', "\underline{A}^T"'] \end{tikzcd}.\]
In this section, we compute this more concretely. In particular, we will show that an affine hypertoric variety $Y_A(0)$ satisfies $\pi_1(Y_A(0)_{\reg})=0$ if and only if $Y_A(0)$ is simple (Corollary \ref{cor:converse}). In the latter part, we also mention an (possible) alternative way to compute $\pi_1(Y_A(0)_{\reg})$ more efficiently (cf.\ Remark \ref{rem:alternative}). 

For $\overline{B}^T : \Z^s\to \Z^{n-d}$, we consider a map $\widetilde{B}^T : \Z^s\to \Z^{n-d}$:  
\[\overline{B}^T=\begin{pmatrix}&&&\\\bm{b^{(1)}}&\bm{b^{(2)}}&\cdots&\bm{b^{(s)}}\\&&&\end{pmatrix}, \ \ \ \widetilde{B}^T:=\begin{pmatrix}&&&\\m_1\bm{b^{(1)}}&m_2\bm{b^{(2)}}&\cdots&m_s\bm{b^{(s)}}\\&&&\end{pmatrix},\]
where $m_k:=\prod_{i\neq k}\ell_i$. Then, for the natural surjection $q: \Z^s\to \Gamma:=\prod_{k=1}^s{\Z/\ell_k\Z}$, by the surjectivity of $\overline{B}^T$ and the definition of $\widetilde{B}^T$, we have $\widetilde{B}^T(\Ker q)=\ell_1\cdots\ell_s\Z^{n-d}$. This induces the following commutative diagram:  
\[\begin{tikzcd}
\Z^s\ar[d, "q"]\ar[r, "\widetilde{B}^T"]&\Z^{n-d}\ar[d]\\
\prod_{k=1}^s{\Z/\ell_k\Z}\ar[r, "h"]&\Z^{n-d}/\ell_1\cdots\ell_s\Z^{n-d}
\end{tikzcd}.\]
Then we have the following description of $\pi_1(Y_A(0)_{\reg})$.  
\begin{proposition}\label{prop:explicit} In the above notation, we have 
\[\Gamma\cap\T_\C^{d-(n-s)}=q(\Ker \widetilde{B}^T)\]
In particular, we have 
\[\pi_1(Y_A(0)_{\reg})\cong\Gamma/\Gamma\cap\T_\C^{d-(n-s)}\cong\Image h=\Image \widetilde{B}^T/\ell_1\cdots\ell_s\Z^{n-d}.\]
\end{proposition}
As an easy corollary, we can give the desired converse of Corollary \ref{cor:hypertoric trivial} as the following．
\begin{corollary}\label{cor:converse}
The following are equivaelnt: 
\begin{itemize}
\item[(i)] $\codim \Sing(Y_A(0))\geq4$. 
\item[(ii)] $\pi_1(Y_A(0)_{\reg})=0$.
\item[(iii)] $Y_A(0) (\text{or} \ \calH_B^0)$ is simple, i.e., $\ell_k=1$ for any $k$. 
\end{itemize}
\end{corollary}
\begin{proof}
(i) $\Leftrightarrow$ (iii) was proven in Corollary \ref{cor:codim2sing} (1). (ii) $\Leftarrow$ (iii)  was already done in Corollary \ref{cor:hypertoric trivial}. We prove 
(ii) $\Rightarrow$ (iii). Since the fundamental group $\pi_1(Y_A(0)_{\reg})=\Image \widetilde{B}^T/\ell_1\cdots\ell_s\Z^{n-d}$ is trivial, we have $\Image \widetilde{B}^T\subset\ell_1\cdots\ell_s\Z^{n-d}$. If there exists $k$ such that $\ell_k>1$, then $m_k\bm{b^{(k)}}\in\ell_1\cdots\ell_s\Z^{n-d}$, so $\bm{b^{(k)}}\in\ell_k\Z^{n-d}$. However this contradicts to the unimodularity of $B$ unless $\bm{b^{(k)}}=0$. This completes the proof (cf.\ Assumption \ref{ass:1}). 
\end{proof}
\begin{remark}(``(i) $\Rightarrow$ (ii)'' is nontrivial among general conical symplectic varieties)\label{rem:nontrivialgeneral}\hspace{2pt}\\
By the above corollary, if $Y_A(0)$ has codimension 2 singular locus, then the fundamental group of the regular locus is nontrivial. This is a very special property of affine hypertoric varieties among conical symplectic varieties. In fact, for example, the nilpotent orbit closure $\overline{\calO}_{[4, 1]}$ in $\gersl_5(\C)$ of type $[4, 1]$ has $A_{2}$-singularity $S_{A_{2}}$ along the codimension 2 singular locus $\calO_{[3, 2]}$ (for example, see the proof of \cite[Proposition 3.2]{KP:minGL}), however its regular locus ${\calO}_{[4, 1]}$ is simply-connected (cf.\ \cite[Corollary 6.1.6]{CM}). 
\end{remark}

\begin{proof}[\textbf{Proof of Proposition \ref{prop:explicit}}]
$(\zeta_{\ell_1}^{k_1}, \ldots, \zeta_{\ell_s}^{k_s})$ is in $\T_\C^{d-(n-s)}=\Ker (\overline{B}^T : \T_\C^s \twoheadrightarrow \T_\C^{n-d})$ if and only if 
\[\overline{B}^T\left( \fr<k_1/\ell_1>, \ldots, \fr<k_s/\ell_s>\right)^T\in\Z^{n-d},\]
where $\zeta_\ell:=e^{\fr<2\pi\sqrt{-1}/\ell>}$. 
This is equivalent to the existence of $p_1, \ldots, p_{n-d}\in\Z$ such that $\widetilde{B}^T(k_1, \ldots, k_s)^T=\ell_1\cdots\ell_s(p_1, \ldots, p_{n-d})^T$. One can easily show that this holds if and only if there exists $q_1, \ldots, q_s\in\Z$ such that $\widetilde{B}^T(k_1-\ell_1q_1, \ldots, k_s-\ell_sq_s)^T=0$. In fact, the ``if'' part is trivial. To prove the ``only if'' part, by interchanging column vectors of ${B}^T$ and multiplying $B^T$ by some $C\in GL_{n-d}(\Z)$ from left, one can assume that $\bm{b^{(1)}}=\bm{e_1}, \ldots, \bm{b^{(n-d)}}=\bm{e_{n-d}}$, where $\bm{e_i}$ is the $i$-th unit vector. Then, $(q_1, \ldots, q_s)=(p_1, \ldots, p_{n-d}, 1, \ldots, 1)$ satisfies the condition. This completes the proof. 
\end{proof}

We give a nontrivial example. 

\begin{example}(The case of $\overline{\calO}^{\min}(\ell_1, \ell_2, \ell_3)$ cf. Example \ref{ex:Ominlll})\\
In this case, by some easy computations, we have 
\[\widetilde{B}^T=\begin{pmatrix}\ell_2\ell_3&0&-\ell_1\ell_2\\0&\ell_1\ell_3&-\ell_1\ell_2\end{pmatrix}, \ \ \ \Ker(\widetilde{B}^T)=\left(\fr<\ell_1/g>, \fr<\ell_2/g>, \fr<\ell_3/g>\right)\Z,\]
where $g:=\gcd(\ell_1, \ell_2, \ell_3)$. 
Thus, by the Proposition \ref{prop:explicit}, we have 
\[\pi_1(\overline{\calO}^{\min}(\ell_1, \ell_2, \ell_3)_{\reg})\cong\left(\Z/\ell_1\Z\times\Z/\ell_2\Z\times\Z/\ell_3\Z\right)\bigg/\left(\overline{\fr<\ell_1/g>}, \overline{\fr<\ell_2/g>}, \overline{\fr<\ell_3/g>}\right)\Z.\]
In particular, we have
\[|\pi_1(\overline{\calO}^{\min}(\ell_1, \ell_2, \ell_3)_{\reg})|=\fr<\ell_1\ell_2\ell_3/g>.\]
More generally, we can compute $\pi_1(\overline{\calO}^{\min}(\ell_1, \ldots, \ell_s)_{\reg})$ as the following:  
\[\pi_1(\overline{\calO}^{\min}(\ell_1, \ldots, \ell_s)_{\reg})\cong\prod_{k=1}^s{\Z/\ell_k\Z}\bigg/\left(\fr<\ell_1/g>, \ldots, \fr<\ell_s/g>\right),\]
where $g:=\gcd(\ell_1, \ldots, \ell_s)$
\end{example}
\begin{remark}
Unfortunately, the concrete expression of $\pi_1(Y_A(0)_{\reg})$ in Proposition \ref{prop:explicit} is not practical to compute efficiently since we should solve some Diophantine equations $\widetilde{B}^T\bm{v}=0$.  
\end{remark}
\begin{remark}(Another description of $\pi_1(Y_A(0)_{\reg})$ as a cokernel)\label{rem:alternative}\\
By considering $Y_A(0)$ as the hyperk\"{a}hler quotient of $\mathbb{H}^n:=(\R\oplus\C)^n$ by the real torus $\T_\R^d$ (cf.\ \cite{BD}), we can describe the fundamental group $\pi_1(Y_A(0)_{\reg})$ in a different way. First, by the remaining $\T_\R^{n-d}$-action on $Y_A(0)$, we can consider the hyperk\"{a}hler moment map   
\[\Phi : Y_A(0) \to \R^{n-d}\otimes\R^3=\R^{3(n-d)}\]
(for the detail, see \cite[Theorem 3.1]{BD}). Then, for the associated hyperplane arrangement $\calH_B^0$ (see Definition \ref{def:arr})，it is known that $\Phi| : \Phi^{-1}(U) \to U:=\R^{3(n-d)}\hspace{-2pt}\setminus\hspace{-2pt}\left(\bigcup_{H\in\calH_B^0} {H\hspace{-2pt} \otimes \hspace{-2pt}\R^3}\right)$ is a $\T_\R^{n-d}$-bundle. Moreover, by Corollary \ref{cor:singularlocus} (2) and the correspondence between GIT construction of $Y_A(0)$ and the hyperk\"{a}hler quotient construction of $Y_A(0)$, one can easily check that $Y_A(0)_{\reg}=\Phi^{-1}(U_0)$, where $U_0:=\R^{3(n-d)}\setminus\left(\bigcup_{{F: \text{flat s.t. $|F|\geq\rank F+1$}}}{H_F \otimes \R^3}\right)$. Then, the $\R$-codimension $\codim_{\R, Y_A(0)_{\reg}}{\Phi^{-1}(U)}\geq3$ since the fiber $\Phi^{-1}(x)$ of any generic element $x\in H\otimes\R^3$ for any $H\in\calH_B^0$ is $\T_\R^{n-d-1}$ (cf.\ \cite[Theorem 3.1.(2)]{BD}). 
In particular, we have $\pi_1(Y_A(0)_{\reg})=\pi_1(\Phi^{-1}(U))$. Hence, we get a description of $\pi_1(Y_A(0)_{\reg})$ through the homotopy long exact sequence:   
\[\pi_2(U) \xrightarrow{\partial} \pi_1(\T_\R^{n-d})=\Z^{n-d} \to \pi_1(\Phi^{-1}(U))=\pi_1((Y_A(0))_{\reg}) \to \pi_1(U). \]
Since we have $\pi_1(U)=0$ and $\pi_2(U)=\Z^s$ by \cite[Theorem 7.2.1]{Bj}, we get the following exact sequence: 
\[\Z^s \xrightarrow{\partial} \Z^{n-d} \to \pi_1((Y_A(0))_{\reg}) \to 0.\]
Although we should describe the boundary map $\partial$ in terms of matrices $A$ or $B$, we don't know the answer. 
\end{remark}

\section{The uniqueness of symplectic structures and Bogomolov's decomposition}

A conical symplectic variety $Y$ with a symplectic 2-form $\om$ with weight $\ell$ is called {\it irreducible} if any symplectic 2-form $\sigma$ with weight $\ell$ satisfies $\sigma=c\om$ ($c\in\Cstar$). In \cite{Nam:equi}, Namikawa proposed the following problem which can be seen as the conical symplectic analogue of the Bogomolov's decomposition theorem. 
\begin{problem}{\rm (\cite[Problem 7.3]{Nam:equi}, Bogomolov's decomposition)}\label{problem:bogomolov}\\
For any conical symplectic variety $(Y, \om)$ with $m:=|\pi_1(Y_{\reg})|<\infty$, its universal cover $(\overline{Y}, \overline{\om})$ (cf.\ Definition {\rm \&} Proposition \ref{def and prop:universal cover}) is decomposed into the product 
\[(\overline{Y}, \overline{\om})\cong\prod_{i}{(Y_i, \om_i)}\]
of irreducible conical symplectic varieties $(Y_i, \om_i)$ with weight $m\ell$. 
\end{problem}
In this section, we solve this problem for affine hypertoric varieties. Since we have already seen that the universal cover of $Y_A(0)$ is given by another (simple) hypertoric variety $Y_{\underline{A}}(0)$, we only have to consider such ones. However, we will show that not necessarily simple affine hypertoric varieties also admit a decomposition into some irreducible ones (also see Remark \ref{rem:universal} (1)).   

First, we consider a sufficient condition for $Y_A(0)$ to decompose into smaller ones. 
\begin{definition}\label{def:decomposable}
A $d\times n$-matrix $A$ (or $Y_A(\alpha)$) is {\it decomposable} if one of the following holds: 
\begin{itemize}
\item[(i)] There exists $i$ such that $\bm{a_i}=0$. 
\item[(ii)] $A\sim A_1\oplus A_2$, i.e., there exists $P\in GL_d(\Z)$ and $Q\in GL_n(\Z)$, such that $PAQ=A_1\oplus A_2$ where $Q$ is a composition of two operations:  interchanging some column vectors and multiplying some column vectors by $-1$ (cf.\ Definition \ref{def:operation}).
\end{itemize} 
We call $A$ (or $Y_A(\alpha)$) {\it indecomposable} if $A$ is not decomposable. 
\end{definition}
\begin{remark}
To consider $\C^2$ as an indecomposable hypertoric variety, we also formally think that the $1\times 1$ zero matrix is indecomposable. This definition is natural since  $A$ is indecomposable in this sense if and only if the associated vector matroid $M(A)$ is {\it (2-)connected} (for the detail, see \cite[{Exercise 4.1.7}]{Ox}). 
\end{remark}
\begin{remark}\label{rem:decomposition}
Actually, if (i) (resp. (ii)) satisfies, there exists a natural isomorphism $Y_A(0)\cong Y_{A^{\hat{i}}}(0)\times \C^2$ (resp. $Y_A(0)\cong Y_{A_1}(0)\times Y_{A_2}(0)$), where $A^{\hat{i}}$ is the matrix obtained from $A$ by deleting the $i$-th column vector (cf.\ \cite[{Lemma 2.27}]{Nag}). In particular, since we can assume that $A$ is equivalent to a matrix of the form  
\begin{equation*}\tag{*}
\left(
\begin{array}{ccc|c@{}c@{}cc}
&&&\begin{array}{|cccc|}\hline  &   & &\\&\vcbig{{A_1}}& \\\hline\end{array}&&&\LargeO\\
&&&& \begin{array}{|cccc|}\hline &   &   & \\&\vcbig{{A_2}}&\\\hline\end{array}&& \\
\vcbig{{O_p}}&&& &\ddots &  \\
&&&\LargeO&&&\begin{array}{|cccc|}\hline&&&\\&\vcbig{{A_r}}&\\\hline\end{array} \\ 
\end{array}
\right),\end{equation*}
where each $A_m$ is indecomposable and $O_p$ denotes the $d\times p$-zero matrix, we have a decomposition $(Y_A(\alpha), \om) \cong (\C^{2p}\times\prod_{m=1}^r{Y_{A_m}(\alpha_m)}, \om_0+\sum_{m=1}^r{\om_m})$, where $\alpha=\alpha_1\oplus\cdots\oplus\alpha_r\in\Z^d$ and $\om_m$ is the standard symplectic form on $Y_{A_m}(\alpha_m)_{\reg}$. 
\end{remark}
The main theorem in this section is the following. 
\begin{theorem}{\rm (Bogomolov's decomposition for affine hypertoric varieties)}\label{thm:main final}\\
If $A$ is equivalent to (*), then the following gives a decomposition into irreducible conical symplectic varieties
\[(Y_A(0), \om) \cong \prod_{i=1}^p{(\C^2, dz\wedge dw)}\times\prod_{m=1}^r{(Y_{A_m}(0), \om_m)}.\]
In particular, $A$ is block indecomposable if and only if $Y_A(0)$ is an irreducible conical symplectic variety. In this case, any (not necessarily symplectic) homogeneous 2-form on $Y_A(0)_{\reg}$ with weight 2 is unique up to scalar. 
\end{theorem}

\begin{remark}
Let $(Y, \om)$ be a conical symplectic variety with weight $\ell$. Then, for any other conical symplectic form $\om'$, its weight is also automatically $\ell$ (cf.\ \cite[Lemma 2.1]{Nam:equi}). Thus, the above theorem says that if $A$ is block indecomposable, then any conical symplectic structure with any weight is only $\om$ up to scalar. 
\end{remark}
Recall for any $\alpha$, there is a natural (projective) birational morphism $\pi_\alpha : Y_A(\alpha) \to Y_A(0)$ which is an isomorphism on $\pi_\alpha^{-1}(Y_A(0)_{\reg})$ (cf.\ Remark \ref{rem:hypertoric is symplectic}). Thus, we have the following.   


\begin{corollary}\label{cor:generalhypertoric}
Fix any $\alpha\in\Z^d$. If $A$ is equivalent to (*), then the following gives a decomposition into irreducible conical symplectic varieties
\[(Y_A(\alpha), \om) \cong \prod_{i=1}^p{(\C^2, dz\wedge dw)}\times\prod_{m=1}^r{(Y_{A_m}(\alpha_m), \om_m)}.\]
In particular, $A$ is block indecomposable if and only if $Y_A(\alpha)$ is an irreducible symplectic variety. In this case, any (not necessarily symplectic) homogeneous 2-form on $Y_A(\alpha)_{\reg}$ with weight 2 is unique up to scalar. 

\end{corollary}



In the subsequent subsections, we will study the space of 2-forms on $Y_A(0)_{\reg}$ and give a proof of the above theorem. For the later convenience, we don't assume that $A$ is indecomposable until the latter part of subsection 7.3. 

\subsection{Reduction to simple and smooth cases}\hspace{2pt}\\
In this subsection, we compare the space of 2-forms on $Y_A(0)_{\reg}$ and the space of 2-forms on the simplification $Y_{\underline{A}}(0)_{\reg}$. More precisely, we have the following:  
\begin{lemma}\label{lem:reduction to simple}
Let $\underline{\alpha}$ be generic and $\pi_{\underline{\alpha}} : Y_{\underline{A}}({\underline{\alpha}}) \to Y_{\underline{A}}(0)$ be the resolution morphism. Then, $\pi_{\underline{\alpha}}$ and the universal covering morphism $\vphi : Y_{\underline{A}}(0) \to Y_A(0)$ (cf.\ Theorem \ref{thm:fundamentalgrp}) induce the following diagram: 
\[\begin{tikzcd}
\Gamma_2(Y_{\underline{A}}({\underline{\alpha}}), \Om^2_{Y_{\underline{A}}({\underline{\alpha}})})\ar[d, hook, "\wr"]&\\
\Gamma_2(Y_{\underline{A}}(0)_{\reg}, \Om^2_{Y_{\underline{A}}(0)_{\reg}})\ar[d, equal]&\\
\Gamma_2(\vphi^{-1}(Y_{{A}}(0)_{\reg}), \Om^2_{\vphi^{-1}(Y_{{A}}(0)_{\reg})})&\Gamma_2(Y_A(0)_{\reg}, \Om^2_{Y_A(0)_{\reg}})\ar[l, hook', "\vphi^*"]
\end{tikzcd},\]
where $\Gamma_2(Z, \Om^2_Z)$ denotes the space of homogeneous 2-forms on $Z$ with weight 2. Similarly, if we replace $\Gamma_2$ by $\Gamma$, the same conclusion will hold. 
\end{lemma}
\begin{proof}
The vertical inclusion induced from the isomorphism $\pi_{\underline{\alpha}}|_{\pi_{\underline{\alpha}}^{-1}(Y_{\underline{A}}(0)_{\reg})}$ is an isomorphism. In fact, $\pi_{\underline{\alpha}}$ is a symplectic resolution, so it is semi-small (\cite[Lemma 2.7]{Kalstratify}). In particular, $\codim_{Y_{\underline{A}}({\underline{\alpha}})}{\pi_{\underline{\alpha}}^{-1}(\Sing(Y_{\underline{A}}(0)))}\geq \fr<1/2>\codim_{Y_{\underline{A}}(0)}{\Sing(Y_{\underline{A}}(0))}$. Since $Y_{\underline{A}}({\underline{\alpha}})$ is simple, we have $\codim_{Y_{\underline{A}}(0)}{\Sing(Y_{\underline{A}}(0))}\geq4$ (cf.\ Corollary \ref{cor:codim2sing} (1)). Thus, we have $\codim_{Y_{\underline{A}}({\underline{\alpha}})}{\pi_{\underline{\alpha}}^{-1}(\Sing(Y_{\underline{A}}(0)))}\geq2$. This proves that the vertical inclusion is an isomorphism. The vertical equality follows from $\codim_{Y_{\underline{A}(0)_{\reg}}}(Y_{\underline{A}(0)_{\reg}}-\vphi^{-1}(Y_A(0)_{\reg}))\geq2$ (cf.\ Remark \ref{rem:phicodim}). 
The horizontal injection is obvious since the right hand side is identified with the invariant part of the left hand side under the fundamental group action since $\vphi|$ is a quotient map by the free action.   
\end{proof}

In the next subsection, we will show that any 2-form on $Y_{\underline{A}}({\underline{\alpha}})$ is induced from some 2-form on $\mu_{\underline{A}}^{-1}(0)$. 

\subsection{Extension of a 2-form on $Y_A(\alpha)$ to a 2-form on $\mu^{-1}(0)$} \hspace{2pt}\\
In this subsection, for a given homogeneous 2-form $\sigma$ on $Y_A(\alpha)$, we consider when we can extend (uniquely) its pull-back $p^*\sigma$ on $\mu^{-1}(0)^{\alpha-st}$ to a homogeneous 2-form on $\mu^{-1}(0)$, where $p : \mu^{-1}(0)^{\alpha-st} \to Y_A(\alpha)$ is the free $\T_\C^d$-quotient morphism. One can see that $p$ and the natural inclusion $\mu^{-1}(0)^{\alpha-st}\subset \mu^{-1}(0)$ induces the following diagram: 
\begin{equation*}\tag{*}
\begin{tikzcd}
\Gamma(Y_A(\alpha), \Om^2_{Y_A(\alpha)})\ar[r, hook]&\Gamma(\mu^{-1}(0)^{\alpha-st}, p^*\Om^2_{Y_A(\alpha)})\ar[r, hook]&\Gamma(\mu^{-1}(0)^{\alpha-st}, \Om^2_{\mu^{-1}(0)^{\alpha-st}})\\
&&\Gamma(\mu^{-1}(0), \Om^2_{\mu^{-1}(0)})\ar[u, "r"].
\end{tikzcd}
\end{equation*}
Actually, 
 the first horizontal homomorphism is injective since $\Om^2_{Y_A(\alpha)}$ is locally free. The second horizontal homomorphism is injective since the smoothness of $p$ implies that the natural exact sequence  
\[\begin{tikzcd}p^*\Om^1_{Y_A(\alpha)} \ar[r]& \Om^1_{\mu^{-1}(0)^{\alpha-st}}\ar[r]& \Om^1_{\mu^{-1}(0)^{\alpha-st}/Y_A(\alpha)}\ar[r]&0
\end{tikzcd}\]
is also left exact, in particular, its wedge also is injective.  Moreover, these natural homomorphisms are commutative with conical $\Cstar$-action, so the degree 2 part of this diagram satisfies the same properties: 
\begin{equation*}\tag{**}
\begin{tikzcd}
\Gamma_2(Y_A(\alpha), \Om^2_{Y_A(\alpha)})\ar[r, hook]&\Gamma_2(\mu^{-1}(0)^{\alpha-st}, p^*\Om^2_{Y_A(\alpha)})\ar[r, hook]&\Gamma_2(\mu^{-1}(0)^{\alpha-st}, \Om^2_{\mu^{-1}(0)^{\alpha-st}})\\
&&\Gamma_2(\mu^{-1}(0), \Om^2_{\mu^{-1}(0)})\ar[u, "r_2"].
\end{tikzcd}
\end{equation*}

In this setting, the following is the main result in this subsection. 
\begin{proposition}\label{prop:keyreflexive}\hspace{2pt}
\begin{itemize}
\item[(1)] In the above diagrams (*) and (**), $r$ and $r_2$ are injective.  

\item[(2)] Moreover, if $Y_A(0)$ is simple, then $r$ and $r_2$ are isomorphisms. 
\item[(3)] We can describe $\Gamma_2(\mu^{-1}(0), \Om^2_{\mu^{-1}(0)})$ as the following:  
\[\Gamma_2(\mu^{-1}(0), \Om^2_{\mu^{-1}(0)})=\bigoplus_{1\leq k<\ell \leq n}{\C dz_k\wedge dz_\ell}\oplus\bigoplus_{1\leq k<\ell \leq n}{\C dw_k\wedge dw_\ell}\oplus\bigoplus_{1\leq j, j' \leq n}{\C dz_j\wedge dw_{j'}}.\]
\end{itemize}
\end{proposition}

To prove this Proposition, we need some lemmas. 
In \cite[Lemma 2.12, the proof of Proposition 2.13]{Nag}, we noted $\mu^{-1}(0)$ is a normal irreducible complete intersection, in particular, the codimension of the singular locus of $\mu^{-1}(0)$ is at least 2 (also, see \cite[section 6]{HS}, \cite[Lemma 4.7]{BK}). We can prove a more refined version as the following. Recall we assume that $\bm{b_j}\neq0$ for any $j=1, \ldots ,n$ (cf.\ Assumption \ref{ass:1})
\begin{proposition}\label{prop:codim}\hspace{2pt}

\begin{itemize}
\item[(1)]$\mu^{-1}(0)$ is a normal irreducible complete intersection of dim $2n-d$. 
\item[(2)]We have $\codim \Sing(\mu^{-1}(0))\geq3$. Moreover, if $B$ is simple, then we have $\codim \Sing(\mu^{-1}(0))\geq4$.
\end{itemize}
\end{proposition}
\begin{proof}
(1) \ This is already proved in, for example,  \cite[section 6]{HS}, \cite[Lemma 4.7]{BK}, and  \cite[{Lemma 2.12}, the proof of Proposition 2.13]{Nag}. 

(2) \  If $(\bm{p}, \bm{q})\in\Sing({\mu}^{-1}(0))$, then the tangent map $d\mu_{(\bm{p}, \bm{q})}=(q_1\bm{a_1}, \ldots ,q_n\bm{a_n}, p_1\bm{a_1}, \ldots ,p_n\bm{a_n})$ is not surjective. However $A$ is surjective, so there exists some $i_0$ such that $p_{i_0}=q_{i_0}=0$. By the assumption $\bm{b_{i_0}}\neq0$ and Gale duality (cf. Lemma \ref{lem:Gale}), $A^{(i_0)} : \Z^{n-1} \to \Z^d$ is also surjective, where $A^{(i_0)}$ is the submatrix of $A$ obtained by deleting the $i_0$-th columns. Thus, by the same argument, there exists $i_1\neq i_0$ such that $p_{i_1}=q_{i_1}=0$. In other words, 
\[(\bm{p}, \bm{q})\in(\mu^{-1}(0))\cap\{z_{i_0}=w_{i_0}=0\}\cap\{z_{i_1}=w_{i_1}=0\}=(\mu^{(i_0, i_1)})^{-1}(0),\]
where we identify as $\C^{2(n-2)}=\C^{2n}\cap\{z_{i_0}=w_{i_0}=0\}\cap\{z_{i_1}=w_{i_1}=0\}$, and $\mu^{(i_0, i_1)} : \C^{2(n-2)} \to V_\C^{(i_0, i_1)}:=V^{(i_0, i_1)}\otimes_\Z \C$ is the moment map associated to the surjection $A^{(i_0, i_1)} : \Z^{n-2} \twoheadrightarrow V^{(i_0, i_1)}:=\Span (\bm{a_j} \ | \ j\neq i_0, i_1)$. Then, by (1), we have
\[\dim (\mu^{(i_0, i_1)})^{-1}(0)=2(n-2)-\dim V_\C^{(i_0, i_1)}.\] 
Since we have $\dim V_\C^{(i_0, i_1)}\geq d-1$, we deduce that $\codim \Sing(\mu^{-1}(0))\geq3$. 
Moreover if $B$ is simple, then $\bm{b_{i_0}}$ and $\bm{b_{i_1}}$ are not parallel vectors, i.e., $\dim \Span(\bm{b_{i_0}}, \bm{b_{i_1}})=2$. By Gale duality (cf. Lemma \ref{lem:Gale}), this implies that $\dim V_\C^{(i_0, i_1)}=\dim \Span (\bm{a_j} \ | \ j\neq i_0, i_1)=d$, i.e., $V_\C^{(i_0, i_1)}=\C^d$. Then, by the same argument above, we have $\codim \Sing(\mu^{-1}(0))\geq4$.   
\end{proof}

To prove Proposition \ref{prop:keyreflexive} (1) and (2), the following criterion due to Vetter will be the key tool.  
 
\begin{theorem}{\rm (\cite{Vetter})}\label{thm:Vetter}\\
Let $Z$ be a normal irreducible local complete intersection. Then,  
\begin{itemize}
\item[(1)]$\Om_Z^p$ : torsion free \ $\Leftrightarrow$ \ $\min\{\codim_Z{\Sing(Z)}, \dim Z\}> p$.
\item[(2)]$\Om_Z^p$ : reflexive \ $\Leftrightarrow$ \ $\min\{\codim_Z{\Sing(Z)}, \dim Z\}> p+1$.
\end{itemize}
\end{theorem}

Finally, we need the following lemma. 

\begin{lemma}
When $Y_A(0)$ is simple, for any $\alpha\in\Z^d$ (not necessarily generic), the $\alpha$-unstable locus $\mu^{-1}(0)^{\alpha-us}:=\mu^{-1}(0)-\mu^{-1}(0)^{\alpha-ss}$ satisfies  
\[\codim_{\mu^{-1}(0)}\mu^{-1}(0)^{\alpha-us}\geq2.\]
\end{lemma}
\begin{proof}
The idea is similar to the proof of the above Proposition. We use the same notation as that proof except the coordinates of $\C^{2n}$ denoted by $(z_1^{(+)}, \ldots, z_n^{(+)}, z_1^{(-)}, \ldots, z_n^{(-)})$. First, recall the description of the unstable locus 
\[(\mu^{-1}(0))^{\alpha-us}=\Set{(\bm{z^{(+)}}, \bm{z^{(-)}}) \in \mu^{-1}(0) \ | \ \alpha \notin \sum_{i: z^{(+)}_i \neq0}{\R_{\geq0}\bm{a_i}} + \sum_{i:z^{(-)}_i\neq0}{\R_{\geq0}(-\bm{a_i})}}.\]
In particular, if $(\bm{p^{(+)}}, \bm{p^{(-)}})\in(\mu^{-1}(0))^{\alpha-us}$, then 
\[\alpha\notin \sum_{j \ : \ p_j^{(+)}p_j^{(-)}\neq0}{\R\bm{a_j}}\varsubsetneq\R^d.\]
Then, by the similar argument as the proof of the above Proposition, simpleness of $B$ implies the surjectivity of $A^{(i_0, i_1)}$ for any different $i_0, i_1$. Hence, there exists $i_0\neq i_1$ and $\vep_0, \vep_1\in\{+, -\}$ such that 
\[(\bm{p^{(+)}}, \bm{p^{(-)}})\in\mu^{-1}(0)\cap\{z_{i_0}^{(\vep_0)}=z_{i_1}^{(\vep_1)}=0\}=\mu_{A^{(i_0, i_1)}}^{-1}(0)\times\C_{z_{i_0}^{(-\vep_0)}}\times\C_{z_{i_1}^{(-\vep_1)}},\]
where $\C_{z_{i_0}^{(-\vep_0)}}$ and $\C_{z_{i_1}^{(-\vep_1)}}$ are the one dimensional vector space corresponding to the coordinates. 
Since the right hand side is of dimension $2n-4-d+2=2n-d-2$, this shows that $\codim_{\mu^{-1}(0)}\mu^{-1}(0)^{\alpha-us}\geq2$.
\end{proof}
\begin{remark}\hspace{2pt}\\
\vspace{-5mm}
\begin{itemize}
\item[(1)] 
The above Proposition and Lemma also hold for $\mu^{-1}(\xi)$ for any $\xi\in\C^d$ by the same argument. 
\item[(2)]
If $Y_A(0)$ is not simple, the codimension of the unstable locus can be one. For example, consider when $Y_A(0)$ is the $A_1$-type surface singularity (cf. Example \ref{ex:Atypesurface}). If we take $\alpha>0$, then,  
\[\mu^{-1}(0)^{\alpha-us}=\{z_1w_1+z_2w_2=0\}\cap\{z_1=z_2=0\}=\{z_1=z_2=0\}. \]
\end{itemize}
\end{remark}

\begin{proof}[\textbf{Proof of Proposition \ref{prop:keyreflexive}}]
(1) \ Since $\mu^{-1}(0)$ is a normal complete intersection irreducible variety with $\codim \Sing(\mu^{-1}(0))\geq3$ by Proposition \ref{prop:codim}, the injectivity of $r$ and $r_2$ follows from the torsion-freeness of $\Om^2_{\mu^{-1}(0)}$ which is a consequence of Theorem \ref{thm:Vetter} (1). 

(2) \ By the assumption and Proposition \ref{prop:codim}, $\mu^{-1}(0)$ is a normal complete intersection irreducible variety with $\codim \Sing(\mu^{-1}(0))\geq4$. Then by the Vetter's theorem \ref{thm:Vetter} (2), $\Om^2_{\mu^{-1}(0)}$ is a reflexive sheaf. On the other hand, the simpleness of $Y_A(0)$ implies that $\codim_{\mu^{-1}(0)}\mu^{-1}(0)^{\alpha-us}\geq2$ by the above lemma. Then, by the property of reflexive sheaves, $r$ and $r_2$ are isomorphisms. 

(3) \ This is just a computation. Actually, we have
\[\Gamma(\mu^{-1}(0), \Om^1_{\mu^{-1}(0)})\cong \fr<{\bigoplus_{j=1}^n{\C[\bm{z}, \bm{w}]dz_j}\oplus\bigoplus_{j=1}^n{\C[\bm{z}, \bm{w}]dw_j}}/{\langle \sum_{j=1}^n{z_jw_j{a_{ij}}} \rangle_{i=1, \ldots, d}+\langle\sum_{j=1}^n{w_j a_{ij}dz_j+z_j a_{ij}dw_j}\rangle_{i=1, \ldots, d}}>.\]
Then, by the weight 2 condition, the relation ideal above cannot contribute to $\Gamma_2(\mu^{-1}(0), \Om^2_{\mu^{-1}(0)})$. Thus, we have the claim.  
\end{proof}

\subsection{The horizontal condition on 2-forms on $\mu^{-1}(0)$}\hspace{2pt}\\
In this subsection, we will show the following theorem. 
\begin{theorem}\label{thm:simpleindecomposable} Let $\alpha$ be generic. If $Y_A(\alpha)$ is simple, and $A$ is block indecomposable, then we have  
\[\Gamma_2(Y_A(\alpha), \Om^2_{Y_A(\alpha)})=\C\om,\]
where $\om$ is the standard conical symplectic form on $Y_A(\alpha)$. 
\end{theorem}
Using this theorem, we can show Theorem \ref{thm:main final} as the following. 
\begin{proof}[\textbf{Proof of Theorem \ref{thm:main final}}]
We only have to show that if $A$ is block indecomposable, then $\Gamma_2(Y_A(0)_{\reg}, \Om^2_{Y_A(0)_{\reg}})=\C\om$, where $\om$ is the standard symplectic 2-form on $Y_A(0)_{\reg}$. By Lemma \ref{lem:reduction to simple} and $\vphi^*\om=\underline{\om}$, we only have to show that $\Gamma_2(Y_{\underline{A}}({\underline{\alpha}}), \Om^2_{Y_{\underline{A}}({\underline{\alpha}})})=\C\underline{\om}$, where $\underline{\om}$ is the standard symplectic 2-form. This follows from the above theorem.  
\end{proof}

First, we consider which 2-form $\om$ on $\mu^{-1}(0)^{\alpha-st}$ descends to a 2-form on $Y_A(\alpha)$ in general. 
Now, we note the following well-known fact. 

\begin{theorem}{\rm (\cite[Theorem 1]{Brion})}\label{fact:horizontal}\\
Let $Z$ be a smooth affine variety with an action of some reductive group $G$, and let $p : Z \to Z/\hspace{-3pt}/G$ be the quotient morphism. If $Z/\hspace{-3pt}/G$ is smooth, then $p^* : \Om^q({Z/\hspace{-3pt}/G}) \to \Om^q_{\text{hor}}(Z)^G$ is an isomorphism for any $q$, where $\Om^q_{\text{hor}}(Z)^G$ is the space of $G$-invariant horizontal $q$-forms, that is, their interior product with any vector field induced by the $G$-action is zero. 
\end{theorem}
Note that we have the same conclusion for $p : \mu^{-1}(0)^{\alpha-st} \to Y_A(\alpha)$ since $\mu^{-1}(0)^{\alpha-st}=\mu^{-1}(0)^{\alpha-ss}$ is covered by $\T_\C^d$-invariant affine open subsets (cf.\ Remark \ref{rem:semistable affine}). By Proposition \ref{prop:keyreflexive} (3) in the previous subsection, we should only consider a 2-form $\sigma$ on $\mu^{-1}(0)^{\alpha-st}$ of the form 
\[\sigma=\sum_{k<\ell}{\beta_{k\ell}dz_k\wedge dz_\ell}+\sum_{k<\ell}{\beta'_{k\ell}dw_k\wedge dw_\ell}+\sum_{j, j'}{\gamma_{j, j'}dz_j\wedge dw_{j'}}.\]
Below, by interchanging cloumn vectors of $A$ and multiplying them by $\pm1$, we can assume and define that 
\[A=\left(\begin{array}{ccc|cc|ccc}
&&&&&&&\\
\bm{a^{(0)}}&\cdots&\bm{a^{(0)}}&\cdots&\cdots&\bm{a^{(r)}}&\cdots&\bm{a^{(r)}}\\
&&&&&&&
\end{array}\right), \ \ \ \ 
\overline{A}:=\begin{pmatrix}&&&\\\bm{{a}^{(1)}}&\bm{{a}^{(2)}}&\cdots&\bm{{a}^{(r)}}\\&&&\end{pmatrix},\]
where $\bm{a^{(j)}}$ and $\bm{a^{(j')}}$ are not parallel vectors if $j\neq j'$, and $\bm{a^{(0)}}=\bm{0}$. We denote by $[n]=J_0\sqcup J_1\sqcup\cdots\sqcup J_s$ the corresponding partition of indices set of column vectors of $A$ (note that $J_0$ might be empty). 
 
At first, a necessary condition on $\sigma$ to descend to a 2-form on $Y_A(\alpha)$ is $\T_\C^d$-invariance by Theorem \ref{fact:horizontal}. Since a $\bm{t}\in\T_\C^d$ acts on $dz_j$ and $dw_j$ as $\bm{t}^*dz_j=\bm{t}^{\bm{a_j}}dz_j$ and $\bm{t}^*dw_j=\bm{t}^{\bm{-a_j}}dw_j$, $\T_\C^d$-invariance of $\sigma$ implies that
\begin{align*}
\sigma&=\sum_{\bm{a}_k=-\bm{a}_\ell}{\beta_{k\ell}dz_k\wedge dz_\ell}+\sum_{\bm{a}_k=-\bm{a}_\ell}{\beta'_{k\ell}dw_k\wedge dw_\ell}+\sum_{\bm{a_{j}}=\bm{a_{j'}}}{\gamma_{j, j'}dz_{j}\wedge dw_{j'}}\\
&=\sigma_0+\sum_{\bm{a_{j}}=\bm{a_{j'}}}{\gamma_{j, j'}dz_{j}\wedge dw_{j'}}\\
&=\sigma_0+\sum_{m=1}^r{\sum_{(j, j')\in J_m\times J_m}{\gamma_{j, j'}dz_{j}\wedge dw_{j'}}}, 
\end{align*}
where $\sigma_0\in\Gamma_2(\C^{2|J_0|}, \Om^2_{\C^{2|J_0|}})$ and the second equality holds since there doesn't exist $1\leq k, \ell\leq s$ such that $\bm{a_k}=-\bm{a_\ell}$ by the form of $A$. 

Next, we consider the horizontal condition when $A$ is indecomposable. In particular, $\sigma_0=0$. Below, to be concise, we assume that $\alpha\in\sum_{j=1}^n{\R_{\geq0}\bm{a_j}}$ (the general case is done similarly). The following lemma gives a proof of Theorem \ref{thm:simpleindecomposable}.


\begin{lemma}In the above setting, a $\T_\C^d$-invariant 2-form $\sigma=\sum_{m=1}^r{\sum_{(j, j')\in J_m\times J_m}{\gamma_{j, j'}dz_{j}\wedge dw_{j'}}}$ on $\mu^{-1}(0)^{\alpha-st}$ is horizontal if and only if the following 2 conditions hold: 
\begin{itemize}
\item[(1)] $\gamma_{1, 1}=\cdots=\gamma_{n, n}$. 
\item[(2)] $\gamma_{j, j'}=0$ if $j\neq j'$.  
\end{itemize}
In other words, $\sigma=\gamma\sum_{j=1}^n{dz_j\wedge dw_j}$ for some $\gamma\in\C$. 
\end{lemma}
\begin{proof}
It is enough to show the ``only if'' part. To show (1) and (2), we apply the horizontal condition to a suitable point $(p, q)\in\mu^{-1}(0)^{\alpha-st}$. Recall the description of the $\alpha$-stable locus: 
\[\mu^{-1}(0)^{\alpha-st}=\Set{(\bm{z}, \bm{w}) \in \mu^{-1}(0) \ | \ \alpha \in \sum_{j: z_j \neq0}{\R_{\geq0}\bm{a_j}} + \sum_{j: w_j\neq0}{\R_{\geq0}(-\bm{a_j})}}.\]
Then, by the assumption $\alpha\in\sum_{j=1}^n{\R_{\geq0}\bm{a_j}}$, for any $j_1\in J_1, \ldots, j_s\in J_s$, $(p, q):=(\bm{e_{j_1}}+\cdots+\bm{e_{j_s}}, \bm{0})$ is contained in $\mu^{-1}(0)^{\alpha-st}$. Then, by the definition, we have 
\[T_{(p, q)}\mu^{-1}(0)^{\alpha-st}=\Set{(\bm{u}, \bm{v})\in T_{(p, q)}\C^{2n} \ | \ \sum_{m=1}^r{v_{j_m}\bm{a_{j_m}}}=\sum_{m=1}^r{v_{j_m}\bm{a^{(m)}}}=0},\]
\[T_{(p, q)}(\T_\C^d\cdot(p, q))=\Span\Set{(a_{i, j_1}\bm{e_{j_1}}+\cdots+a_{i, j_r}\bm{e_{j_r}}, \bm{0})^T \ | \ i=1, \ldots, d},\]
\[\sigma_{(p, q)}((\bm{u}, \bm{v}), (\bm{\xi}, \bm{\eta}))=\sum_{m=1}^r{\sum_{(j, j')\in J_m\times J_m}{\gamma_{j, j'}(u_{j}\eta_{j'}-v_{j'}\xi_{j})}}. \]
Thus, $\sigma$ is horizontal at $(p, q)$, i.e., $\sigma_{(p, q)}((\bm{u}, \bm{v}), (\bm{\xi}, \bm{\eta}))=0$ for any $(\bm{u}, \bm{v})\in T_{(p, q)}\mu^{-1}(0)^{\alpha-st}$ and $(\bm{\xi}, \bm{\eta})\in T_{(p, q)}(\T_\C^d\cdot(p, q))$ if and only if for any $v_{j'}$ ($j' \in J_m\setminus\{j_m\}$) and any $v_{j_m}$'s such that $\sum_{m=1}^r{v_{j_m}\bm{a^{(m)}}}=0$, we have
\begin{equation*}\tag{*}
\bm{0}=-\sum_{m=1}^r{\left(\sum_{j'\in J_m}\gamma_{j_m, j'}v_{j'}\right)\bm{a^{(m)}}}.
\end{equation*}

(1) \ In particular, if we take as $v_{j'}=0$ for any $j'\in J_m\setminus\{j_m\}$, then for any element $(v_{j_1}, \ldots, v_{j_r})^T\in\Ker \overline{A}$, we have  
\[\bm{0}=\sum_{m=1}^r{\gamma_{j_m, j_m}v_{j_m}\bm{a_{j_m}}}=\sum_{m=1}^r{\gamma_{j_m, j_m}v_{j_m}\bm{a^{(m)}}}.\]
Then, by the indecomposability of $A$ and two lemmas below, we have $\gamma_{j_1, j_1}=\cdots=\gamma_{j_r, j_r}$. 

(2) \ For any $j\neq j'\in J_m$, we take $j_1\in J_1, \ldots, j_r\in J_r$ as $j=j_m$. If we consider the horizontal condition (*) for $v_{j''}=0$ ($\forall j''\in[n]\setminus\{j'\}$), then $(v_{j_1}, \ldots, v_{j_r})^T=\bm{0}\in\Ker \overline{A}$ and  
\[\bm{0}=-\gamma_{j_m, j'}v_{j'}\bm{a^{(m)}}.\]
This implies that $\gamma_{j, j'}=0$. This completes the proof. 
\end{proof}

Below, we will show two lemmas which are used in the proof of the above Proposition. 
\begin{lemma}{\rm (A special case of \cite[Exercise 4.1.5]{Ox})}\\
Let $A$ be a $d\times n$-matrix ($d\leq n$). 
$A$ is indecomposable if and only if $n=1$ or its simplification $\overline{A}$ is indecomposable and $A$ has no zero column vectors (Don't confuse $\underline{A}$ with $\overline{A}$).   

\end{lemma}
\begin{lemma}The following are equivalent.  

\begin{itemize}
\item[(i)]$A$ is decomposable (cf.\ Definition \ref{def:decomposable}) 
\item[(ii)]There exists a partition $I\sqcup J=[n]$ of the indices of column vectors of $A$ such that $\Span(\bm{a_i} \ | \ i\in I)\cap\Span(\bm{a_j} \ | \ j\in J)=0$. 
\item[(iii)]There exists $\bm{\gamma}=(\gamma_1, \ldots, \gamma_n)\neq\C(1, 1, \ldots, 1)$ $(\gamma_i\neq0 \ \forall i)$ such that if $A\bm{v}=\bm{0}$ then $A(\gamma_1v_1, \ldots, \gamma_nv_n)^T=\bm{0}$. 
\end{itemize}
\end{lemma}
\begin{proof}
It is easy to show that (i) and (ii) are equivalent (For example, see \cite[Exercise 4.1.7]{Ox}). Thus, we only have to show that those two conditions are equivalent to (iii). (i)$\Rightarrow$ (iii) is easy. We prove (iii)$\Rightarrow$(ii). By interchanging the coordinates, we can assume that $\bm{\gamma}=(\gamma_1, \ldots, \gamma_1, \gamma_{s+1}, \ldots, \gamma_n)$ ($s<n$), where for any $i> s$, $\gamma_i\neq\gamma_1$. Then, we show that the partition $I:=\{1, \ldots, s\}$ and $J:=\{s+1, \ldots, n\}$ satisifies $W_I\cap W_J=0$, where $W_I:=\Span(\bm{a_i} \ | \ i\in I)$ and $W_J$ is similar. Assume that there exists $0\neq\bm{w}\in W_I\cap W_J$. Then, we can express $\bm{w}$ as 
\[\bm{w}=\sum_{i\in I}{c_i\bm{a_i}}=\sum_{j\in J'}{c_j\bm{a_j}},\]
where $J'\subset J$ is a fixed subset of $J$ such that $\{\bm{a_j} \ | \ j\in J'\}$ are linearly independent. In particular, $\bm{v}:=(c_1, \ldots, c_s, -c_{s+1}, \ldots, -c_n)^T$ satisfies $A\bm{v}=0$. By the assumption of (iii), we have $A(\gamma_1c_1, \ldots, \gamma_1c_s, -\gamma_{s+1}c_{s+1}, \ldots, -\gamma_nc_n)^T=\bm{0}$. Then, 
\[\bm{0}\neq\gamma_1\bm{w}=\sum_{i\in I}{\gamma_1c_i\bm{a_i}}=\sum_{j\in J'}{\gamma_{j}c_j\bm{a_j}}.\]
However, this is a contradiction since $\{\bm{a_j} \ | \ j\in J'\}$ are linearly independent and $\gamma_1\neq \gamma_j$ for any $j\in J'\subset J$. This completes the proof. 
\end{proof}
\section{All symplectic strctures on hypertoric varieties and its application to classification}

In the previous section, for indecomposable hypertoric varieties $Y_A(\alpha)$, we showed that they admit only one homogeneous 2-form on them up to scalar (cf.\ Corollary \ref{cor:generalhypertoric}). In this section, we describe the space of homogeneous 2-forms on general decomposable  hypertoric varieties (cf. Proposition \ref{prop:generalform}). 
As a corollary, we can show that for any conical symplectic 2-form $\sigma$ on $Y_A(\alpha)$, there exists a $\psi\in\Aut^{\Cstar}(Y_A(\alpha))$ such that $\psi^*\sigma=\om$. In particular, as a byproduct, we will show that if two affine or smooth hypertoric variets are $\Cstar$-equivariant isomorphic as algebraic varieties, then they are $\Cstar\times\T_\C^{n-d}$-equivariant isomorphic as symplectic varieties (cf. Corollary \ref{cor:classification}). This gives a refinement of the result in \cite[Theorem 4.2]{Nag}.  


For a given matrix $A$, we can assume that $A$ is of the form 
\[A=
\left(
\begin{array}{ccc|c@{}c@{}cc}
&&&\begin{array}{|cccc|}\hline  &   & &\\&\vcbig{{A_1}}& \\\hline\end{array}&&&\LargeO\\
&&&& \begin{array}{|cccc|}\hline &   &   & \\&\vcbig{{A_2}}&\\\hline\end{array}&& \\
\vcbig{{O_p}}&&& &\ddots &  \\
&&&\LargeO&&&\begin{array}{|cccc|}\hline&&&\\&\vcbig{{A_r}}&\\\hline\end{array} \\ 
\end{array}
\right),\]
where each $A_m$ is indecomposable and $O_p$ denotes the $d\times p$-zero matrix. In what follows, $[n]=I_0\sqcup I_1\sqcup\cdots\sqcup I_r$ denotes the corresponding decomposition of indices set of column vectors. Then, as we have seen in Remark \ref{rem:decomposition}, we had a natural isomorphism $Y_A(\alpha) \cong \C^{2p}\times\prod_{m=1}^r{Y_{A_m}(\alpha_m)}$, where $\alpha=\alpha_1\oplus\cdots\oplus\alpha_r\in\Z^d$. Then, we can describe the space of $\Cstar$-equivariant 2-forms on $Y_A(\alpha)_{\reg}$.

\begin{proposition}\label{prop:generalform}
In the above setting, for any (not necessarily generic) $\alpha=\alpha_1\oplus\cdots\oplus\alpha_r\in\Z^d$, we can describe the space of $\Cstar$-equivariant 2-forms with weight 2 as the following: 
\begin{align*}
\Gamma_2(Y_{{A}}({{\alpha}})_{\reg}, \Om^2_{Y_{{A}}({{\alpha}})_{\reg}})=\Gamma_2(\C^{2p}, \Om^2_{\C^{2p}})\oplus\bigoplus_{m=1}^r{{\C{\om}_m}},
\end{align*}
where $\om_m$ is the standard conical symplectic form on $Y_{A_m}(\alpha_m)_{\reg}$. In particular, any conical symplectic form on $Y_A(\alpha)$ is of the form 
\[g^*\om_0+\sum_{m=1}^r{\gamma_m {\om}_m} \ \ \ (g\in GL(\C^{2p}), \ \gamma_m\in\Cstar),\]
where $\om_0$ is the standard symplectic structure on $\C^{2p}$. 
\end{proposition}
\begin{proof}
First, we assume that $Y_A(\alpha)$ is simple and $\alpha$ is generic. As noted in Proposition \ref{prop:keyreflexive} (2), by the simpleness, any 2-form in $\Gamma_2(Y_A(\alpha), \Om_{Y_A(\alpha)})$ is induced from a 2-form in $\Gamma_2(\mu_A^{-1}(0), \Om^2_{\mu_A^{-1}(0)})$. Since $\mu_A^{-1}(0)$ is an affine variety, we have 
\[\Gamma_2(\mu_A^{-1}(0), \Om^2_{\mu_A^{-1}(0)})=\bigoplus_{m=0}^r{\Gamma_2(\mu_{A_m}^{-1}(0), \Om^2_{\mu_{A_m}^{-1}(0)})}\oplus\bigoplus_{m\neq m'}{\Gamma_1(\mu_{A_{m}}^{-1}(0), \Om^1_{\mu_{A_m}^{-1}(0)})\otimes\Gamma_1(\mu_{A_{m'}}^{-1}(0), \Om^1_{\mu_{A_{m'}}^{-1}(0)})}.\] 
Now, by the similar argument as the proof of Proposition \ref{prop:keyreflexive} (3), we have 
\[\Gamma_1(\mu_{A_{m}}^{-1}(0), \Om^1_{\mu_{A_m}^{-1}(0)})=\bigoplus_{j\in I_m}{\C dz_j}\oplus\bigoplus_{j\in I_m}{\C dw_j}. \]
However, since we have $\bm{a_j}\neq\pm\bm{a_{j'}}$ for any $j\in I_m$, $j'\in I_{m'}$ ($m\neq m'$), any nonzero 2-form in $\bigoplus_{m\neq m'}{\Gamma_1(\mu_{A_{m}}^{-1}(0), \Om^1_{\mu_{A_m}^{-1}(0)})\otimes\Gamma_1(\mu_{A_{m'}}^{-1}(0), \Om^1_{\mu_{A_{m'}}^{-1}(0)})}$ is not $\T_\C^d$-invariant. In particular, such a 2-form will not descend to a 2-form on $Y_A(\alpha)$ (see the argument after Theorem \ref{fact:horizontal}). Thus, we only have to consider when a 2-form in $\bigoplus_{m=0}^r{\Gamma_2(\mu_{A_m}^{-1}(0), \Om^2_{\mu_{A_m}^{-1}(0)})}$ descends to a 2-form on $Y_A(\alpha)=\C^{2p}\times\prod_{m=1}^r{Y_{A_m}(\alpha_m)}$. Then, by Theorem \ref{thm:simpleindecomposable}, for $1\leq m\leq r$, we have 
\[\Gamma_2(Y_{A_m}(\alpha_m), \Om^2_{Y_{A_m}(\alpha_m)})=\C\om_m.\]

Next, we consider the general case, i.e., $Y_A(\alpha)$ is not necessarily simple or $\alpha$ is not generic. By comparing with the simplification, we have the following diagram: 
\[\begin{tikzcd}
Y_{\underline{A}}({\underline{\alpha}})\ar[d, "\pi_{\underline{\alpha}}"]&Y_A(\alpha)\ar[d, "\pi_\alpha"]\\
Y_{\underline{A}}(0)\ar[r, "\vphi"]&Y_A(0)
\end{tikzcd},\]
where we take a generic $\underline{\alpha}$. 
Then, by the argument in subsection 7.1, we have the following diagram: 
\[\begin{tikzcd}
\Gamma_2(\C^{2p}, \Om^2_{\C^{2p}})\oplus\bigoplus_{m=1}^r{{\C{\underline{\om}}_m}}\ar[d, equal]&\Gamma_2(\C^{2p}, \Om^2_{\C^{2p}})\oplus\bigoplus_{m=1}^r{{\C{\om}_m}}\ar[dd, hook]\\
\Gamma_2(Y_{\underline{A}}({\underline{\alpha}}), \Om^2_{Y_{\underline{A}}({\underline{\alpha}})})\ar[d, hook, "\wr"]&\\
\Gamma_2(Y_{\underline{A}}(0)_{\reg}, \Om^2_{Y_{\underline{A}}(0)_{\reg}})\ar[d, equal]&\Gamma_2(Y_A(\alpha)_{\reg}, \Om^2_{Y_A(\alpha)_{\reg}})\ar[d, hook]\\
\Gamma_2(\vphi^{-1}(Y_{{A}}(0)_{\reg}), \Om^2_{\vphi^{-1}(Y_{{A}}(0)_{\reg})})&\Gamma_2(Y_A(0)_{\reg}, \Om^2_{Y_A(0)_{\reg}})\ar[l, hook', "\vphi^*"]
\end{tikzcd}\]
Now, note that $\vphi^*$ is an identity on $\Gamma_2(\C^{2p}, \Om^2_{\C^{2p}})$ and sends the standard symplectic 2-form $\om_m$ to the standard symplectic 2-form $\underline{\om}_m$. Consequently, in the above diagram, each homomorphism is an isomorphism. This completes the proof of the first part of the statement. For the latter part, we only have to show that any (conical) symplectic 2-form on $\C^{2p}$ is given by $g^*\om_0$ for some $g\in GL(\C^{2p})$. This follows from the Darboux's theorem.   

\end{proof}


\begin{remark}\label{rem:auto}
For any conical symplectic 2-form $\om=g^*\om_0+\sum_{m=1}^r{\gamma_m\om_m}$ on a hypertoric variety $Y_A(\alpha)=\C^{2p}\times\prod_{m=1}^r{Y_{A_m}(\alpha_m)}$, we can take a $\Cstar$-equivariant automorphism 
\[\psi:=(g, {\sqrt{\gamma_1}}, \ldots, {\sqrt{\gamma_r}}) : Y_A(\alpha) \xrightarrow{\sim} Y_A(\alpha),\]
where ${\sqrt{\gamma_m}}\in \Aut^{\Cstar}(Y_{A_m}(\alpha_m))$ is the automorphism induced from the scalar multiplication by $\sqrt{\gamma_m}$. 
By general results \cite[Theorem 3.1]{Nam:equi}, for a conical symplectic variety $(Y, \om)$ and any conical symplectic form $\sigma$ on $Y$, there exists a $\Cstar$-equivariant automorphism $\psi : Y \to Y$ such that $\psi^*\sigma=\psi$. The above $\psi$ gives a such automorphism concretely for affine hypertoric varieties. 
\end{remark}

From now on, we consider an application of Proposition \ref{prop:generalform} to the classification of affine (or smooth) hypertoric varieties. In \cite[Theorem 4.2]{Nag}, we proved the following. 

\begin{theorem}{\rm (\cite[Theorem 4.2, Remark 4.3]{Nag})}\label{thm:classification}\\
For any two affine hypertoric varieties $Y_A(0)$ and $Y_{A'}(0)$, 
the following are equivalent: 
\begin{itemize}
\item[(1)] $Y_A(0)$ and $Y_{A'}(0)$ are $\Cstar$-equivariant isomorphic as symplectic varieties. 
\item[(2)] $(Y_A(0), \om)$ and $(Y_{A'}(0), \om')$ are $\Cstar\times\T_\C^{n-d}$-equivariant isomorphic as symplectic varieties, where $\T_\C^{n-d}$-action is the remaining Hamiltonian torus action. 
\end{itemize}
The same statement holds for any two smooth hypertoric varieties. 
\end{theorem}
\begin{remark}(Combinatorial classification)\label{rem:combinatorial}\\
For two affine hypertoric varieties $Y_A(0)$ and $Y_{A'}(0)$, the above conditions are equivalent to $A\sim A'$ (cf.\ \cite[Theorem 4.2]{Nag}). In other words, $M(A)\cong M(A')$, where $M(A)$ is the associated vector matroid. Similarly, for two smooth ones $Y_A(\alpha)$ and $Y_{A'}(\alpha')$, the above conditions are equivalent to that the hyperplane arrangements $\calH_B^{\alpha}$ and $\calH_{B'}^{\alpha'}$ define the same {\it zonotope tiling} in the sense of \cite{AP}. 
\end{remark}

Now, by Remark \ref{rem:auto}, if two affine (resp. smooth) hypertoric varieties are $\Cstar$-equivariant isomorphic to each other, then the isomorphism can be replaced by an $\Cstar$-equivariant isomorphism as symplectic varieties. Thus, we have the following refined version of the above theorem. 
\begin{corollary}\label{cor:classification}
For any two affine hypertoric varieties $Y_A(0)$ and $Y_{A'}(0)$, 
the following are equivalent: 
\begin{itemize}
\item[(1)] $Y_A(0)$ and $Y_{A'}(0)$ are $\Cstar$-equivariant isomorphic as algebraic varieties. 
\item[(2)] $(Y_A(0), \om)$ and $(Y_{A'}(0), \om')$ are $\Cstar\times\T_\C^{n-d}$-equivariant isomorphic as symplectic varieties, where $\T_\C^{n-d}$-action is the remaining Hamiltonian torus action. 
\end{itemize}
The same statement holds for any two smooth hypertoric varieties.   
\end{corollary}

\begin{remark}
This corollary can be seen as the hypertoric analogue of Berchtold's theorem \cite{Ber} on toric varieties, which says that for any two toric varieties, they are isomorphic as abstract varieties if and only if they are isomorphic as toric varieties. 
\end{remark}

By Remark \ref{rem:combinatorial}, this corollary says that the classification of the associated combinatorial objects $A$ (resp. $\calH_B^{\alpha}$) give the classification of hypertoric varieties $Y_A(0)$ (resp. $Y_A(\alpha)$) up to $\Cstar$-equivariant isomorphisms as (abstract) algebraic varieties. 






\end{document}